\documentclass[10pt]{amsart}
    \pdfoutput=1
\usepackage[maxbibnames=99, style=alphabetic, maxalphanames=99, minalphanames=6, 
sorting=nyt, giveninits=true]{biblatex}
\AtEveryBibitem{
    \clearfield{issn}
  \clearfield{isbn}
  \clearfield{doi}
  \clearlist{location}
  \clearfield{month}
  \clearfield{day}
  \clearfield{series}
  \clearlist{language}
  \clearlist{translator}
  \clearfield{URL}
  \clearfield{url}
  \clearfield{eprint}
  \renewbibmacro{in:}{}
}
    \usepackage{graphicx}
    \usepackage[english]{babel}
    \usepackage{graphicx}
    \usepackage{framed}
    \usepackage[normalem]{ulem}
    \usepackage{amsmath}
    \usepackage{dynkin-diagrams}
    \usepackage{amsthm}
    \usepackage{amssymb}
    \usepackage{hyperref} 
    \usepackage{comment}
    \usepackage{tikz}
    \usetikzlibrary{matrix,calc}
    \usepackage{mathtools}
    \usepackage{amsfonts}
    \usepackage{enumerate}
    \usepackage[utf8]{inputenc}
    \usepackage[top=1 in,bottom=1in, left=1 in, right=1 in]{geometry}
    \usepackage[capitalize]{cleveref}
    \usepackage[colorinlistoftodos]{todonotes}
    \usepackage[all]{xy}
    \usepackage{mathrsfs}
    \usepackage{longtable} 
    \usepackage{stmaryrd}

    \newcommand{\Z}{\mathbb{Z}} 
    
    \newcommand{\LG}{\mathfrak{g}}
    
    \newcommand{\LT}{\mathfrak{t}}

    \newcommand{\End}{\operatorname{End}} 
    \newcommand{\Hom}{\operatorname{Hom}}
    
    \renewcommand{\Hom}{\operatorname{Hom}} 
    \newcommand{\uHom}{\underline{\operatorname{Hom}}} 
    \renewcommand{\End}{\operatorname{End}} 
    \newcommand{\uEnd}{\underline{\operatorname{End}}} 
    \newcommand{\QCoh}{\operatorname{QCoh}}
    \newcommand{\IndCoh}{\operatorname{IndCoh}}
    
    \newcommand{\C}{\mathcal{C}} 
    \renewcommand{\O}{\mathcal{O}} 
    \newcommand{\F}{\mathcal{F}} 
    \newcommand{\D}{\mathcal{D}} 
    \renewcommand{\\}{\backslash}
    
    \usepackage[mmddyyyy]{datetime}

    \theoremstyle{definition}
    \newtheorem{Theorem}{Theorem}[section]
    \newtheorem{Conjecture}[Theorem]{Conjecture}
    \newtheorem{Corollary}[Theorem]{Corollary}
    \newtheorem{Definition}[Theorem]{Definition}
    \newtheorem{Proposition}[Theorem]{Proposition}
    \newtheorem{Remark}[Theorem]{Remark}
    \newtheorem{Example}[Theorem]{Example}
    \newtheorem{Lemma}[Theorem]{Lemma}

    \title{Classification of Nondegenerate $G$-Categories}
    \author[Tom Gannon]{Tom Gannon \\ Appendix with Germ\'an Stefanich}

    \newcommand{\CatTwTw}{\mathcal{D}(N\backslash G/N)^{\text{(}T \times T\text{, w)}}_{\text{nondeg}}}
    \newcommand{\CatTwTwplus}{\mathcal{D}(N\backslash G/N)^{\text{(}T \times T\text{, w), +}}_{\text{nondeg}}}

    \newcommand{\Hpsi}{\mathcal{H}_{\psi}}

    \newcommand{\AvN}{\text{Av}_*^N}
    \newcommand{\Avpsi}{\text{Av}_{!}^{\psi}}
    \newcommand{\Symt}{\text{Sym(}\mathfrak{t}\text{)}}
    
    \newcommand{\LTd}{\LT^{\ast}}

    \newcommand{\Wext}{\tilde{W}^{\text{aff}}}
    \newcommand{\Waff}{W^{\text{aff}}}

    \newcommand{\IndCohx}{\text{IndCoh}(\LTd \times_{\LTd\sslash\Wext} \LTd)}
    
    \newcommand{\AvNShiftedAndLifted}{\widetilde{\mathsf{A}}_*}

    \newcommand{\DGCatContk}{\text{DGCat}^k_{\text{cont}}}

\newcommand{\AvNshifted}{\AvN[\text{dim}(N)]}
\newcommand{\Avpsishifted}{\Avpsi[-\text{dim}(N)]}

\newcommand{\indsch}{\mathcal{X}}
\newcommand{\characterlatticeforT}{X^{\bullet}(T)}

\newcommand{\Spec}{\text{Spec}}

\newcommand{\FourierMukai}{\text{FMuk}}

\newcommand{\generalstacktoGITquotientmap}{\phi}

\newcommand{\Hpsiliteral}{\mathcal{H}_{\psi}}

\newcommand{\Res}{\text{Res}}

\newcommand{\parabolicrestrictionLIFTED}{\text{WRes}}
\newcommand{\horocycleFunctor}{\text{hc}}

\newcommand{\quotientmapforcoarsequotient}{\overline{s}}
\newcommand{\quotientmaptostackquotient}{q}

\newcommand{\terminalmapfromC}{\alpha}

\newcommand{\tildeV}{\tilde{\mathbb{V}}}

\usepackage{tikz-cd}
\newtheorem{Notation}[Theorem]{Notation}
\DeclareMathOperator{\Groth}{Groth}
\DeclareMathOperator{\DGroth}{DGroth}
\DeclareMathOperator{\Funct}{Funct}
\newcommand{\ccal}{{\mathcal{C}}}
\newcommand{\dcal}{{\mathcal{D}}}
\newcommand{\proj}{{\normalfont \text{proj}}}
\newcommand{\op}{{\normalfont \text{op}}}

\begin{document}
    \maketitle
    \begin{abstract}
We classify a \lq dense open\rq{} subset of categories with an action of a reductive group, which we call nondegenerate categories, entirely in terms of the root datum of the group. As an application of our methods, we also: 

\begin{enumerate}
    \item Upgrade an equivalence of Ginzburg and Lonergan, which identifies the category of bi-Whittaker $\D$-modules on a reductive group with the category of $\Wext$-equivariant sheaves on a dual Cartan subalgebra $\LTd$ which descend to the coarse quotient $\mathfrak{t}^*\sslash \Wext$, to a monoidal equivalence (showing that the Whittaker-Hecke category is symmetric monoidal and answering a question of Drinfeld) and
    \item Show the parabolic restriction of a very central sheaf acquires a Weyl group equivariant structure such that the associated equivariant sheaf descends to the coarse quotient $\mathfrak{t}^*\sslash \Wext$, proving a modified conjecture of Ben-Zvi--Gunningham.
\end{enumerate}

    \end{abstract}
    
\tableofcontents 

    \section{Introduction}
    \subsection{Main Results}\label{SurveyOfResults}
    Much of modern geometric representation theory can be interpreted as the study of groups acting on \textit{categories} and the natural symmetries that the various invariants obtain; we will see specific examples of this in \cref{MotivationAndSurvey}. Therefore it is of natural interest to study the class of all categories with an action of a split reductive group $G$. Our main theorem in this paper provides a \lq coherent\rq{} description for a localized class of $G$-categories known as \textit{nondegenerate categories}.
    
    \begin{Definition}\label{Nondegenerate G Category Definition for Simply Connected Group}
    Assume $G$ is simply connected. A \textit{nondegenerate} $G$-category is a $G$-category $\C$ such that for every rank one parabolic $P_{\alpha}$, the invariants $\C^{[P_{\alpha}, P_{\alpha}]}$ vanish. 
    \end{Definition}

    We study some of the basic properties of nondegenerate $G$-categories in the companion paper \cite{GannonNew}.  For example, we argue that any $G$-category $\C$ admits a functor $\C \to \C_{\text{nondeg}}$ that, informally speaking, has the same properties as the map $j^!: \D(X) \to \D(U)$ for an open subset $j: U \xhookrightarrow{} X$. This is made precise in \cite[Section 2.4.5]{GannonNew}.
    
Our main result in this paper states that the 2-category of nondegenerate $G$-categories admits a coherent description as modules over sheaves on an ind-scheme $\Gamma_{\Wext}$ defined only in terms of the action of the extended affine Weyl group $\Wext := \characterlatticeforT \rtimes W$ on $\LTd$. As ind-schemes, we have $\Gamma_{\Wext} \simeq \pi_1(G^{\vee}) \times \Gamma_{\Waff}$, where $G^{\vee}$ denotes the Langlands dual group and $\Gamma_{\Waff}$ denotes the union of graphs in $\LTd \times \LTd$ given by the $\Waff$-action on $\LTd$. One can identify $\Gamma_{\Wext} \simeq \LTd \times_{\LTd\sslash\Wext} \LTd$ for a certain prestack $\LTd\sslash\Wext$ known as the \textit{coarse quotient}, a main object of study of \cite{GannonDescentToTheCoarseQuotientForPseudoreflectionAndAffineWeylGroups}. This implies that one can use the convolution formalism of \cite[Section 5.5]{GaRoI} to equip $\IndCoh(\Gamma_{\Wext})$ with a monoidal structure. Our main result can be summarized as follows:

\begin{Theorem}\label{NondegGCatsTheorem}
There is an equivalence of 2-categories 
\[G\text{-mod}_{\text{nondeg}} \simeq \IndCoh(\Gamma_{\Wext})\text{-mod}\]

\noindent where the left hand side denotes the 2-category of all nondegenerate $G$-categories.
\end{Theorem}
\raggedbottom
As we will review in \cref{A Universal Nondeg Intro Section}, this result can reinterpreted as an equivalence of monoidal categories $\D(N\backslash G/N)^{T \times T, w}_{\text{nondeg}} \simeq \IndCoh(\Gamma_{\Wext})$; see \cref{Main Monoidal Equivalences} for the full statement. 

A key difference in the theory of groups $G$ acting on categories $\C$ from the theory of groups acting on vector spaces is the existence of nontrivial morphisms between morphisms of invariants $\C^{H_1}$ and $\C^{H_2}$, where each $H_i$ is a closed subgroup of $G$; concretely, these are manifested by adjoint functors. These relations will prove a key technical tool in the proof of our theorem
. In particular, we will see  in \cref{N Averaging Fully Faithful on Whittaker Subcategory} that, for any $G$-category, the category $\C^N_{\text{nondeg}}$ admits a $W$-action, and that there is a fully faithful functor $\C^{N^{-}, \psi} \xhookrightarrow{} \C_{\text{nondeg}}^{N, W}$ via constructing a certain left adjoint. 

When $\C$ itself is given by Whittaker $\D$-modules on $G$, one can show \cite[Corollary 3.4]{GannonNew} that $\C^N \simeq \C^N_{\text{nondeg}}$, and, using this special case of \cref{N Averaging Fully Faithful on Whittaker Subcategory}, in \cref{Mellin Transform for biWhittaker Sheaves} we derive a monoidal equivalence between the category of bi-Whittaker $\D$-modules on $G$ and $\Wext$-equivariant sheaves on $\LTd$ which descend to the coarse quotient, providing a monoidal upgrade of \cite{Gin} and \cite{Lo}. Using nondegeneracy, we also show that the parabolic restriciton of a very central $\D$-module acquires a $W$-equivariant structure such that the sheaf (with its equivariance) descends to the coarse quotient $\LTd\sslash\Wext$, see \cref{ParabolicInductionAndRestriction}. 

\subsection{Motivation and Survey of Known Results}\label{MotivationAndSurvey}
    Assume we are given a finite dimensional vector space $V$ over an algebraically closed field $k$, equipped with an endomorphism $T: V \to V$. A familiar paradigm in representation theory and algebraic geometry is to regard $V$ as a module over the ring $k[x]$, where $x$ acts by the transformation $T$, and to write $V$ as a direct sum of its generalized eigenspaces $V_{\alpha}$. Furthermore, the vector space $V$ can be recovered from the various $V_{\alpha}$. We may equivalently view $V$ as a sheaf over the line, and then each $V_{\alpha}$ can be identified as the subsheaf which lives over $\alpha$. This particular example gives the well known Jordan normal form of a matrix, but there are analogues of this process for any $k$-algebra $A$ and any module $M \in \QCoh(\text{Spec}(A))$. 
    
    We can also apply this idea to other representation theoretic contexts. For example, let $\LG$ be a semisimple Lie algebra, and let $M$ be a representation of the Lie algebra. Then it is known (see eg \cite{HumO}) that $Z\LG$ is a polynomial algebra, and furthermore that we may identify $\text{Spec}(Z\LG) \simeq \LTd\sslash W$. Therefore we may spectrally decompose a given $\LG$-representation by viewing it as a sheaf on the space $\LTd\sslash W$.  
    
    We will discuss analogues for these results one categorical level higher. Specifically, our notion of vector space will be replaced with that of a category. The analogue of an algebraic group acting on a vector space is a \textit{group acting on a category}. For example, if $G$ acts on a variety $X$, then the category $\D(X)$, the category of $\D$-modules on $X$, obtains a canonical $G$-action. Similarly, we can obtain a $G$-action on the category $\LG\text{-mod}$. 
    
    Analogous to the case of a group acting on a vector space, we can define the invariants of a group acting on a category. For example, one can understand representations of the Lie algebra $\LG$ of a semisimple algebraic group $G$ via the invariants $\text{Rep}(G) \simeq \LG\text{-mod}^G$, or the associated zero block of the BGG category $\O$, which can be viewed (via the Beilinson-Bernstein localization theorem) as objects of $\D(G/B)^N$. Similarly, one can also study the other blocks of category $\O$ via the \textit{twisted invariants} $\D(G/_{\lambda}B)^N$ for $\lambda \in \LTd$. 
    
    Certain twisted invariants play a special role in geometric representation theory. Specifically, for a reductive group $G$ acting on a category $\C$, one can take the \textit{Whittaker invariants} $\C^{N^{-}, \psi}$. This category can be interpreted as the generically twisted $N^{-}$ invariants of $\C$. Often, Whittaker subcategories can be easier to understand than the usual $N$-invariants. For example, we have seen above that one may identify $\D(G/B)^N$ contains all of the information of the BGG category $\mathcal{O}_{0}$, whereas one can use the ideas of \cite{BBM} discussed below to show that $\D(G/B)^{N^{-}, \psi} \simeq \text{Vect}$.
    
    The Whittaker invariants of a category have often been used to \lq bootstrap\rq{} information about the original category, see, for example, \cite{AB} or  \cite{BezYun}. In fact, our results below can be viewed as an attempt to generalize the work done by \cite{BezYun} at generalized central character 0 to the setting of varying central character. One can formally argue that the category of \textit{bi-}Whittaker invariants of $\D(G)$, denoted $\Hpsi := \D(N^-_{\psi}\backslash G/_{-\psi}N^-)$, acts on the Whittaker invariants of any category with a $G$-action. It is therefore of interest to determine an explicit description for $\Hpsi$. This was identified in terms of sheaves on $\LTd$ which are equivariant with respect to the extended affine Weyl group $\Wext := \characterlatticeforT \rtimes W$ for the Langlands dual group, where $\characterlatticeforT$ is the character lattice $\text{Hom}_{\text{AlgGp}}(T, \mathbb{G}_m)$:
    
    \begin{Theorem}\label{Abelian Categorical Mellin Transform for biWhittaker Sheaves} (\cite{Lo}, \cite{Gin}) There is an equivalence identifying the abelian category of bi-Whittaker $\D$-modules on $G$ with the abelian category of $\Wext$-equivariant quasicoherent sheaves on $\LTd$ which descend to the coarse quotient $\LTd\sslash \Wext$.
    \end{Theorem}
    
    Ginzburg \cite{Gin} and Ben-Zvi--Gunningham \cite[Section 1.2]{BZG} also recorded the expectation that a derived, monoidal variant of \cref{Abelian Categorical Mellin Transform for biWhittaker Sheaves} should hold (see \cref{Conventions Subsection} for our exact categorical conventions). To state this precisely, we first recall the notion of the \textit{Mellin transform}, a symmetric monoidal, $W$-equivariant equivalence $\FourierMukai: \IndCoh(\LTd/\characterlatticeforT) \xrightarrow{\sim} \D(T)$, where the notation follows \cite{LurieEllipticI}. Here, we use ind-coherent sheaves rather than quasi-coherent sheaves since our $\D$-modules are right $\D$-modules in the sense of \cite{GaiRozCrystals}, although since $T$ is smooth, there is also a similar equivalence for left $\D$-modules $\QCoh(\LTd/\characterlatticeforT) \xrightarrow{\sim} \D^{\ell}(T)$.\footnote{Although the Mellin transform is typically stated as an equivalence of abelian categories, as the aforementioned categories can be identified as the derived categories of their respective hearts, one can deduce an analogous symmetric monoidal equivalence of the corresponding $\infty$-categories, which we show in \cref{Mellin Transform Appendix}.}
    
    Letting $\overline{\mathsf{s}}: \LTd/\Wext \to \LTd\sslash\Wext$ denote the canonical map, we can now state the derived, monoidal version of \cref{Abelian Categorical Mellin Transform for biWhittaker Sheaves}:  
    
    
\begin{Theorem}\label{Mellin Transform for biWhittaker Sheaves}
    There is a monoidal, $t$-exact, fully faithful functor \[\jmath^* \AvNShiftedAndLifted: \Hpsi \xhookrightarrow{} \D(T)^W\] such that, under the Mellin transform, this functor induces monoidal equivalence $F'$ for which the following diagram commutes
    \raggedbottom
     \begin{equation*}
  \xymatrix@R+2em@C+2em{
  \IndCoh(\LTd\sslash\Wext) \ar[d]_{F'} \ar[r]^{\mathsf{s}^!} &  \IndCoh(\LTd/\Wext) \ar[d]^{\FourierMukai} \\
  \Hpsi \ar[r]^{\jmath^*\AvNShiftedAndLifted} & \D(T)^W
  }
 \end{equation*} and such that $F'[\text{dim}(\LTd)]$ is $t$-exact. 
\end{Theorem}

\begin{Remark}\label{Our Argument Shows a t-Exact Equivalence of IndCoh on Coarse Quotient and Hpsi}
The heart of any $t$-exact functor of DG categories (or, more generally, triangulated categories) equipped with $t$-structures is an exact functor of abelian categories, see \cite[Proposition 1.3.17]{FaisceauxPerverseBeilinsonBernsteinDeligne}.  Therefore, taking the heart of the equivalence in \cref{Mellin Transform for biWhittaker Sheaves}, our methods show that there is an \textit{exact} equivalence of the abelian categories in \cref{Abelian Categorical Mellin Transform for biWhittaker Sheaves}.
\end{Remark}

\begin{Remark}
The $t$-exactness of $\jmath^*\AvNShiftedAndLifted$ is essentially due to Ginzburg \cite[Theorem 1.5.4]{Gin}. Moreover, as we will see below, the composite 
\raggedbottom
\[\Hpsi \xrightarrow{\jmath^*\AvNShiftedAndLifted} \D(T)^W \xrightarrow{\text{oblv}^W} \D(T)\]

\noindent of $\jmath^*\AvNShiftedAndLifted$ and the forgetful functor $\mathrm{oblv}^W$ can be identified, up to cohomological shift, with the composite of an averaging functor $\AvN$ and a standard equivalence of categories $\jmath^*$ whose construction we recall below in \cref{SupportOfWhittakerSheaves}.
\end{Remark}

Of course, our proof of \cref{Mellin Transform for biWhittaker Sheaves} is different than the proofs of \cite{Gin} or \cite{Lo}. For example, the ideas in \cite{Lo} pass through the geometric Satake equivalence, whereas we do not. We view our proof as closer in spirit to the proof of \cite{Gin}; for example, both use versions of the \textit{Gelfand-Graev action}, see \cite[Section 3]{GannonNew}. However, the idea to use the groupoid $\Gamma_{\Wext}$ is taken from \cite{Lo} and \cite{BZG}. 

\subsection{Symmetric Monoidality of Whittaker-Hecke Category}
Much of our proof of \cref{Mellin Transform for biWhittaker Sheaves} is phrased in the language of categorical representation theory. This can provide another conceptual explanation for \cref{Mellin Transform for biWhittaker Sheaves} which makes certain aspects of this equivalence follow from general, categorical principles. For example, in \cite{Gin}, Ginzburg, following Drinfeld, noted that a derived version of \cref{Abelian Categorical Mellin Transform for biWhittaker Sheaves} would have the following consequence, proved by Ben-Zvi and Gunningham shortly after the first edition of \cite{Gin} was published:

\begin{Corollary}\label{Whittaker Hecke Category is Symmetric Monoidal} \cite[Corollary 6.15]{BZG}
The convolution monoidal stucture on $\Hpsi$ can be upgraded to a symmetric monoidal structure. 
\end{Corollary}

The proof of \cref{Whittaker Hecke Category is Symmetric Monoidal} in \cite{BZG} is somewhat indirect. Specifically, the authors prove that the \textit{cohomologically sheared} (or \textit{asymptotic}) version of the Whittaker-Hecke category $\mathcal{H}_{\psi}^{\hbar}$ is symmetric monoidal using the derived, loop rotation equivariant geometric Satake theorem of \cite{BezrukavnikovFinkelbergEquivariantSatakeCategoryandKostantWhittakerReduction}, and argue that one can obtain a symmetric monoidal structure on $\Hpsi$ by unshearing (see \cite[Section 5.3]{BZG}) a graded lift of $\mathcal{H}_{\psi}^{\hbar}$ provided by a mixed version of the derived, loop rotation equivariant geometric Satake theorem of \cite{BezrukavnikovFinkelbergEquivariantSatakeCategoryandKostantWhittakerReduction}, which is not currently available in the literature.\footnote{For some progress in this direction, see the results announced in \cite{HoLiRevisitingMixedGeometry}.}

The fully faithfulness of \cref{Mellin Transform for biWhittaker Sheaves} provides an alternate proof of the symmetric monoidality of $\Hpsi$ which is more direct. Specifically, because we will see that the functor $\AvNShiftedAndLifted$ is monoidal in \cref{MonoidalityOfAvN}, the fully faithfulness of $\AvNShiftedAndLifted$ immediately implies \cref{Whittaker Hecke Category is Symmetric Monoidal}, since we can identify $\Hpsi$ as a monoidal subcategory of a symmetric monoidal category. 

\subsection{Generalization to Nondegenerate $G$-categories}\label{Generalization to Nondegenerate G-Categories Section} The principles of categorical representation theory will also allow us to prove the following generalization of \cref{Mellin Transform for biWhittaker Sheaves} to all nondegenerate $G$-categories (discussed above in \cref{SurveyOfResults}), see \cref{N Averaging Fully Faithful on Whittaker Subcategory}. Specifically, note that we may interpret the symmetric monoidality of  \cref{Mellin Transform for biWhittaker Sheaves} as a statement regarding spectrally decomposing categories with a $G$-action. For example, \cref{Mellin Transform for biWhittaker Sheaves} says that if $\C$ is a category with a $G$-action, then for each $[\lambda] \in \LTd\sslash\Wext$, we may consider the eigencategories of its Whittaker invariants $(\C^{N^{-}, \psi})_{[\lambda]}$. However, some categories do not admit Whittaker invariants. For example, one can show that $\text{Vect}^{N^{-}, \psi} \simeq 0$, see \cite[Example 2.41]{GannonNew}.

On the other hand, work of \cite{BZGO} (which we will summarize below in \cref{BZGOTheorem}) shows that if $\C$ is any $G$-category, the subcategory $\C^N$ (with its natural symmetries) determines $\C$. It is therefore of interest to relate the $N$-invariants of a category to the Whittaker invariants. To do this, we recall the following well known result, which we provide a proof for the sake of completeness in \cite[Proposition 1.3]{GannonNew}: 

\begin{Proposition}\label{SupportOfWhittakerSheaves}
Restricting to the open Bruhat cell induces an equivalence $$\D(G/N)^{N^-, \psi} \xrightarrow{\sim} \D(N^-B/N)^{N^-, \psi}$$ which is $t$-exact and (right) $T$-equivariant. Moreover, if $\jmath^*$ denotes the composite of this equivalence with the equivalence \begin{equation}\label{Equivalence of Whittaker on Big Cell with DT}\tag{*}\D(N^-B/N)^{N^-, \psi} \simeq \D(T)\end{equation} then $\jmath^*$ exhibits $\D(G/N)^{N^-, \psi}$ as a right $\D(T)$-category that is free of rank one.
\end{Proposition}

Using the equivalence $\jmath^*$ of \cref{SupportOfWhittakerSheaves}, we may reinterpret the statement of \cref{Mellin Transform for biWhittaker Sheaves} in the language of groups acting on categories. Specifically, after applying $\jmath^*$, \cref{Mellin Transform for biWhittaker Sheaves} in fact says that for the left $G$-category $\C := \D(G)^{N^{-}, -\psi}$, the averaging functor $\AvN: \C^{N^{-}, \psi} \to \C^N$, after applying cohomological shift, lifts to a fully faithful, $t$-exact functor:\footnote{Here, and in what follows, we lightly abuse notation and identify $\D(G)^{N} \simeq \D(G/N)$ without explicitly recording this equivalence of categories in our notation.} 
\raggedbottom
\begin{equation}\label{GeneralAvSetup}
\C^{N^-, \psi} \xhookrightarrow{\AvNShiftedAndLifted} \C^{N, W}.
\end{equation}

Now let $\C$ be any $G$-category. Since $\C^N$ determines $\C$, one may ask whether a similar technique can be applied. Unfortunately, for example, in the universal case $\C = \D(G)$, the category $\C^N$ is not expected to admit a natural $W$-action. However, as we have shown in \cite{GannonNew} (and recall below in \cref{Summary of Companion Paper}) $W$ \textit{does} act on any nondegenerate $G$-category $\C$. We will show that, for such $\C$, the analogue of \cref{Mellin Transform for biWhittaker Sheaves} holds:
    
\begin{Theorem}\label{N Averaging Fully Faithful on Whittaker Subcategory}
For any nondegenerate $G$-category $\C$, the category $\C^N$ acquires a canonical $W$-action and there is an induced, fully faithful functor $\AvNShiftedAndLifted: \C^{N^{-}, \psi} \xhookrightarrow{} \C^{N, W}$. \end{Theorem}

\begin{Example}\label{BBMExample}
    Let $\C = \D(G)^{N^{-}, -\psi}$ with its canonical left $G$-action. Then $\C$ is nondegenerate, see \cite[Corollary 3.4]{GannonNew}. Furthermore, by \cref{SupportOfWhittakerSheaves} we have $C^N \simeq \D(T)$ and a result of Ginzburg's (\cite[Proposition 5.5.2]{Gin}) states that this isomorphism is $W$-equivariant. Thus a special case of \cref{N Averaging Fully Faithful on Whittaker Subcategory} gives the fully faithfulness statement in \cref{Mellin Transform for biWhittaker Sheaves}. 
\end{Example}

\subsection{A Universal Nondegenerate $G$-category}\label{A Universal Nondeg Intro Section}
Recall that if $H$ is any algebraic group acting on a category $\C$, we may also define its \textit{weak invariants}. This is defined by forgetting the action of the category $\D(H)$ down to an action of $\IndCoh(H)$ and taking invariants of $\C$ as an $\IndCoh(H)$-module category. The notion of weak invaraints is specific to groups acting on categories (as opposed to vector spaces). Moreover, for any discrete group $F$ the data of a weak action is equivalent to a strong action since the forgetful functor $\text{oblv}: \D(F) \xrightarrow{} \IndCoh(F)$ is an equivalence. 

\begin{Example}\label{Harish-Chandra Example}
The category $\D(G)^{G, w} \simeq \LG\text{-mod}$, while $\D(G)^G \simeq \text{Vect}$. We also note that the category $\D(G)$ obtains two commuting $G$-actions (one from the left action of $G$ on itself and one from the right). Therefore, we may define the category $\D(G)^{G \times G, w}$, and this category identifies with the \textit{Harish-Chandra category} $HC_G$ the category of $U\LG$-bimodules with an integrable diagonal action. Note we also see from this example a natural way to interpret the $G$-action on $\LG\text{-mod}$. 
\end{Example}

\begin{Example}
The category $\LG\text{-mod}$ acquires a $G$-action, and so, in particular, the category $\LG\text{-mod}^N$ acquires a $T \cong B/N$ action. We can identify the category $\LG\text{-mod}^{N, (T,w)}$ with the \textit{universal category} $\mathcal{O}$, see \cite{KalSaf}. We survey and study the connections to the BGG category $\mathcal{O}$ in much more detail in \cite{GannonNew}. In particular, we show there that the left adjoint to the functor $\AvNShiftedAndLifted$ at a fixed central character can be identified with an enhanced version of Soergel's functor $\mathbb{V}$. Thus, as explained in more detail in \cite[Section 1]{GannonNew}, this gives one interpretation of the left adjoint of $\AvNShiftedAndLifted$ in the universal case--it is an analogue of Soergel's $\mathbb{V}$ which does not require a fixed character. 
\end{Example}

The following theorem then states that a category $\C$ with a $G$-action can be recovered from $\C^{N, (T,w)}$ with its natural symmetries. 

\begin{Theorem}\label{BZGOTheorem} \cite[Theorem 1.2]{BZGO}
The monoidal categories $\D(G)$, $\D(N\backslash G/N)$, and $\D(N\backslash G/N)^{T \times T, w}$ are all Morita equivalent. 
\end{Theorem}

\noindent Therefore, to understand results on $G$-categories, it suffices to understand the monoidal category $\D(N\backslash G/N)^{T \times T, w}$. In particular, via application of \cite[Theorem 1.3]{BZGO}, we may similarly understand nondegenerate $G$-categories via understanding the localized monoidal category $\D(N\backslash G/N)_{\text{nondeg}}^{T \times T, w}$. 

We are now in a position to recast \cref{NondegGCatsTheorem} as an equivalence of monoidal categories: 

\begin{Theorem}\label{Main Monoidal Equivalences}
    There are monoidal equivalences of categories
    \raggedbottom
    \begin{equation}\label{First Main Moniodal Equivalence}\D(N\backslash G/N)_{\text{nondeg}}^{T \times T, w} \simeq \IndCoh(\Gamma_{\Wext}) \simeq \IndCoh(\LTd \times_{\LTd\sslash \Wext} \LTd)\end{equation}
    \raggedbottom
    \begin{equation}\label{Second Line of Equivalence}\D(N\backslash G/N)_{\text{nondeg}} \simeq \IndCoh(\LTd/\characterlatticeforT \times_{\LTd\sslash \Wext} \LTd/\characterlatticeforT)\end{equation}
   \noindent which are $t$-exact up to cohomological shift. 
   
\end{Theorem}

In particular, the formalism of \cite[Theorem 1.1]{BZG} applies\footnote{For adjoint $G$, this particular example for the category $\IndCoh(\Gamma_{\Wext})$ is, in fact, given in \cite[Section 2.7.3]{BZG}. The new input is here is providing a description of $\IndCoh(\Gamma_{\Wext})$ in terms of $\D$-modules on $G$, see \cref{Main Monoidal Equivalences}.} and we obtain an $\mathbb{E}_2$ functor $\text{IndCoh}(\LTd\sslash\Wext) \to \mathcal{Z}(\D(G)_{\text{nondeg}}) \simeq \D(G)_{\text{nondeg}}^G$, where given a monoidal category $\mathcal{A}$, $\mathcal{Z}(\mathcal{A})$ denotes its \textit{center} $\mathcal{Z}(\mathcal{A}) := \uEnd_{\mathcal{A} \times \mathcal{A}}(\mathcal{A})$. Using this, we may consider the eigencategories for any nondegenerate $G$-category over the category $\IndCoh(\LTd\sslash\Wext)$, see \cite[Section 2.8.2]{BZG}. 

\begin{Remark}
The statement of \cref{Main Monoidal Equivalences} can be interpreted at the level of abelian categories as follows, which we state for $G$ adjoint type for the ease of exposition. Let $\Gamma_{\Waff}$ be the union of the graphs of the affine Weyl group $\Waff$. Then $\Gamma_{\Waff}$ is an ind-scheme, and so, in particular, every compact object in $\IndCoh(\Gamma_{\Waff})$ can be realized as the pushforward $i_{S, *}^{\IndCoh}(\mathcal{F}_S)$ for $i_S: \Gamma_S \to \Gamma_{\Waff}$ the closed embedding of the union of some finite collection graphs of the affine Weyl group, and $\F_S$ an object of the abelian category of coherent sheaves on $\Gamma_S$ \cite[Chapter 3, Section 1]{GaRoII}. Therefore, every object of $\D(N\backslash G/N)^{T \times T, w, \heartsuit}$ admits a quotient which can be viewed as a filtered colimit of such sheaves. 
\end{Remark}

\begin{Remark}
We can also interpret nondegeneracy as a localization of 2-categories 
\raggedbottom
\[\D(G)\text{-mod} \to \D(G)\text{-mod}_{\text{nondeg}}.\] 

\noindent This perspective may prove useful in the local geometric Langlands correspondence, which studies twisted representations of the \textit{loop group}. Our localization can be interpreted as an upgraded version of the functor $\C \mapsto \text{Whit}(\C)$. The functor Whit is of importance to the local geometric Langlands program, see \cite{Ras2}. However, one does not need knowledge of this program for the results below. 
\end{Remark}

\begin{Remark}
This result, along with \cref{NondegGCatsTheorem}, admits an interpretation in the theory of 2 ind-coherent sheaves, in upcoming work of Arinkin-Gaitsgory and di Fiore-Stefanich \cite{DS}. In this vein, an informal interpretation of \cref{NondegGCatsTheorem} is that we can identify a generic part of $G$-categories as free of rank one over $\IndCoh(\LTd\sslash\Wext)$, and furthermore we have complete understanding of the singular support behavior which can occur. 
\end{Remark}

\begin{Remark}\label{Remark on Ngo Functor of BZG}
In \cite{BZG}, a $\mathbb{E}_2$-functor $\text{Ng\^o}_{\hbar}$ from the cohomologically sheared Whittaker-Hecke category $\Hpsi^{\hbar}$ to the cohomologically sheared category $\mathcal{D}_{\hbar}(G/G)$ is constructed using the derived, loop rotation equivariant geometric Satake of \cite{BezrukavnikovFinkelbergEquivariantSatakeCategoryandKostantWhittakerReduction}. Using this, the authors also sketch an argument that there is an $\mathbb{E}_2$-functor $\text{Ng\^o}: \Hpsi \to \D(G/G)$. Assuming the existence of such an $\mathbb{E}_2$-functor, we would obtain that \textit{all} $G$-categories diagonalize over $\LTd\sslash \Wext$. The idea that such categories with a strong $G$-action should diagonalize over $\LTd\sslash \Wext$ is implicitly used in the proofs below, and was inspired by \cite{BZG}. 
\end{Remark}
\newcommand{\hc}{\text{hc}}
\newcommand{\hcW}{\tilde{hc}}
\newcommand{\ch}{\text{ch}}
\newcommand{\chW}{\tilde{ch}}

\subsection{Equivariance on Very Central $\mathcal{D}$-modules on $G$}\label{ParabolicInductionAndRestriction}
Using the ideas of categorical representation theory, we can also provide some evidence for a recent conjecture of Ben-Zvi and Gunningham on the essential image of enhanced parabolic restriction, which we state explicilty below after recalling some preliminaries.

\subsubsection{The Horocycle Functor and Parabolic Restriction}Consider the category 
\raggedbottom
\[\mathcal{Z}(\D(G)) := \uEnd_{G \times G}(\D(G)) \simeq \D(G)^G_.\] Here, as with all invariants in this subsection, $G$ is acting via the adjoint action. This category is canonically the center of all categories with a $G$-action. Associated to it is a functor known as \textit{parabolic restriction} $\text{Res}: \D(G)^G \to \D(T)^T$. For an excellent survey on parabolic restriction in many of its guises in representation theory, see \cite{KalSaf}. We will define parabolic restriction in terms of a related functor, known as the \textit{horocycle functor} hc, which is defined as the composite:
    \raggedbottom
    \[\D(G)^G \xrightarrow{\text{oblv}^G_B} \D(G)^B \xrightarrow{\text{Av}_*^{N \times N}} \D(N\backslash G/N)^{T}_.\]
    
    \noindent Let $i: N\backslash B/N \xhookrightarrow{} N\backslash G/N$ denote the closed embedding.
    
    \begin{Definition}
    The \textit{parabolic restriction} functor is the composite 
    \raggedbottom
    \[\D(G)^G \xrightarrow{\text{hc}} \D(N\backslash G/N)^T \xrightarrow{i^!} \D(T)_.^T\]
    \end{Definition}
    
    It was proved that parabolic restriction is $t$-exact in \cite{BezYom}. In particular, parabolic restriction induces an exact functor of abelian categories
    \raggedbottom
    \[\text{Res}: \D(G)^{G, \heartsuit} \to \D(T)^{T, \heartsuit}_.\]
    
    \noindent These abelian categories can often be easier to work with than their corresponding derived counterparts. For example, a standard argument (see, for example \cite[Section 10.3]{RaskinAffineBBLocalization}) shows the forgetful functor identifies $\D(G)^{G, \heartsuit}$ as a full abelian subcategory of $\D(G)^{\heartsuit}$ for $G$ any connected algebraic group. On the other hand, $\D(G)^G$ is not the derived category of its heart for any nontrivial reductive $G$, see \cite[Proposition 1.3.3.7, Dual Version]{LuHA} for the particular property which fails. 
    
    Let $\text{Ind}$ denote the left adjoint to parabolic restriction, known as \textit{parabolic induction}. At the level of abelian categories, it was shown in \cite[Section 3.2]{ChenOnTheConjecturesofBravermanKazhdan} that if $\F \in \D(T)^{W, \heartsuit} \simeq \D(T)^{T \rtimes W, \heartsuit}$, then the sheaf 
    \raggedbottom
    \[\text{Ind}(\text{oblv}^W(\F)) \in \D(G)^{G, \heartsuit}\]
    
    \noindent acquires a canonical $W$-representation functorial in $\F$. Using this, it is standard to show that one can lift parabolic restriction to a functor
    \raggedbottom
    \[\parabolicrestrictionLIFTED: \D(G)^{G, \heartsuit} \to \D(T)^{W, \heartsuit}\]
    
    \noindent which Ginzburg computed explicitly and showed identifies $\D(T)^{W, \heartsuit}$ with a quotient category of $\D(G)^{G, \heartsuit}$ in \cite[Theorem 4.4]{GinzburgParabolicInductionandtheHarishChandraDModule}. 
    
    \subsubsection{Very Central $\mathcal{D}$-Modules}
    While parabolic restriction in general has many interesting properties, it is \textit{not} monoidal in general. However, a standard argument (see \labelcref{MonoidalityOfHorocycle} below) gives that the horocycle functor \textit{is} monoidal. This suggests that a distinguished role is played by those sheaves $\F \in \D(G)^{G, \heartsuit}$ for which one can recover $\horocycleFunctor(\F)$ from $\Res(\F)$, which, following \cite{BZG}, we call \textit{very central}:
    
    \begin{Definition}
    We say a sheaf $\F \in \D(G)^{G, \heartsuit}$ is \textit{very central} if $\text{oblv}^T \circ \text{hc}(\F) \in \D(N\backslash G/N)$ is supported on $N\backslash B/N$. 
    \end{Definition}
    
     We let $\mathcal{V}$ denote the category of very central $\D$-modules, which is an abelian category by the $t$-exactness of parabolic restriction. These $\D$-modules on $G$ have recently appeared in the \'etale setting in works of Chen. Specifically, in \cite{ChenAVanishingConjecturetheGLnCase}, the author argues that the acyclicity of $\rho$-Bessel sheaves follows from the very centrality of certain sheaves obtained from enhanced parabolic induction on the \'etale analogue of sheaves in $\D(T)^W$ which descend to the coarse quotient. This very centrality is proved in \cite{ChenOnTheConjecturesofBravermanKazhdan}. 
     
     In \cite{BZG}, the authors conjecture that very central $\D$-modules are precisely those given by the Ng\^o functor (discussed in \cref{Remark on Ngo Functor of BZG}) at the level of abelian categories.\footnote{This also justifies the term \lq very central,\rq{} since the abelian category of very central $\D$-modules is expected to be a \textit{symmetric} monoidal subcategory of an abelian category which is only braided monoidal.} We state the following formulation of the conjecture here: 
    
\begin{Conjecture}\label{BZGConjecture}  \cite[Conjecture 2.14(2)]{BZG} 
The restricted functor of abelian categories 
\raggedbottom
\[\parabolicrestrictionLIFTED: \mathcal{V} \to \D(T)^{W, \heartsuit}\]

\noindent has essential image given by those sheaves descending to the coarse quotient $\LTd\sslash \Wext$. 
\end{Conjecture}
    
In \cref{Horocycle Functor Section}, we provide some evidence for \cref{BZGConjecture}. Specifically, we show: 
    
\begin{Theorem}\label{Parabolic Restriction of a Very Central Sheaf in Heart Has W-Equivariant Structure Descending to Coarse Quotient}
If $\mathcal{F} \in \D(G)^{G, \heartsuit}$ is very central, then there is a $W$-equivariant structure on $\text{Res}(\mathcal{F})$ such that 
\raggedbottom
\[\text{oblv}^T(\text{Res}(\mathcal{F})) \in \D(T)^{W, \heartsuit}\]

\noindent descends to the coarse quotient.
\end{Theorem}

We note that, while nondegeneracy and Whittaker invariants are used heavily in the proof of \cref{Parabolic Restriction of a Very Central Sheaf in Heart Has W-Equivariant Structure Descending to Coarse Quotient}, the statement of \cref{Parabolic Restriction of a Very Central Sheaf in Heart Has W-Equivariant Structure Descending to Coarse Quotient} uses neither. 
\subsection{Outline of Paper}\label{PaperOutline}
In \cref{Preliminaries Section}, we review some conventions and prove a result on comonadicity which will be used later. In \cref{ProofOfFirstTheorem}, using the foundations on nondegenerate $G$-categories developed in the companion paper \cite{GannonNew}, we prove \cref{Mellin Transform for biWhittaker Sheaves} and \cref{N Averaging Fully Faithful on Whittaker Subcategory}. In \cref{ProofOfMainTheorem}, we prove \cref{Main Monoidal Equivalences}. In \cref{The Nondegenerate Horocycle Functor Section}, we then prove \cref{Parabolic Restriction of a Very Central Sheaf in Heart Has W-Equivariant Structure Descending to Coarse Quotient}. This paper also contains one appendix, \cref{Mellin Transform Appendix}, which is written jointly with Germ\'an Stefanich and upgrades the classical Mellin transform to symmetric monoidal equivalence of DG categories.
\subsection{Acknowledgements}
I am especially grateful to my advisor, Sam Raskin, who originally posed the question of whether the ideas of groups acting on categories could be used to extend the results of \cite{Lo} and \cite{Gin} to a general expression of the form \labelcref{GeneralAvSetup}, and provided a tremendous amount of encouragement and useful insights along the way. Additionally, I would like to thank Rok Gregoric for patiently explaining the details of the theory of $\infty$-categories. I would also like to thank David Ben-Zvi, Justin Campbell, Tsao-Hsien Chen, Desmond Coles, Dennis Gaitsgory, Sanath Devalapurkar, Gurbir Dhillon, Victor Ginzburg, Sam Gunningham, Natalie Hollenbaugh, Gus Lonergan, Kendric Schefers, Brian Shin, Germ\'an Stefanich, Harold Williams, and Yixian Wu for many interesting and useful conversations. Additionally, I would like to thank the referee for a careful reading of this paper as well as for useful comments.
 
This project was completed while I was a graduate student at the University of Texas at Austin, and I would like to thank everyone there for contributing to such an excellent environment.  

\section{Preliminaries}\label{Preliminaries Section}

\subsection{Conventions}\label{Conventions Subsection}
We summarize the conventions we use here; we also give a more detailed discussion of our conventions and the motivation behind them in \cite[Section 2.1]{GannonNew}. We let $k$ denote a field of characteristic zero. Unless otherwise stated, we use the term \lq category\rq{} to refer to a DG category, or in other words a $k$-linear presentable stable $\infty$-category. Additionally, unless otherwise stated, we assume all functors between (DG) categories commute with colimits. We will always emphasize when a category we are working with is not a DG category: one prominent example which will play a key role for us is the \textit{eventually coconnective} subcategory $\C^{+} := \cup_{n \in \mathbb{Z}^{\geq 0}}\C^{\geq -n}$ of a (DG) category $\C$ equipped with a $t$-structure, which we view as a non-cocomplete DG category.

\subsection{Comonadicity} We now recall the notion of comonadicity and prove some elementary results related to it which will be used in what follows. 
\subsubsection{Barr-Beck-Lurie} In the proof of \cref{Main Monoidal Equivalences}, we will use the Barr-Beck-Lurie theorem. We will recall the result here for the reader's convenience--this is summarized in much more depth and proved in \cite[Theorem 4.7.3.5]{LuHA}. 

Given a functor of (not necessarily DG) infinity categories $L: \C \to \D$ which admits a right adjoint $R: \D \to \C$, we can obtain a comonad in $\D$ which we denote $LR$. The functor $L$ canonically lifts to a functor $L^{\text{enh}}: \C \to LR\text{-comod}(\D)$. 
    
\begin{Definition}
We say that $L$ is \textit{comonadic} if $L^{enh}$ is an equivalence. 
\end{Definition}
    
\begin{Theorem}\label{BarrBeck}(Barr-Beck-Lurie Theorem for Comonads) The following are equivalent:
\begin{itemize}
    \item The functor $L$ is comonadic.
    \item The functor $L$ is conservative, and moreover for any $L$-split cosimplicial object $C^{\bullet}$ of $\C$, the totalization of $C^{\bullet}$ exists\footnote{Observe that if $\C$ and $\D$ are DG categories then these totalizations always exist. This is because, by definition, any DG category is cocomplete and thus, by the conventions as in \cite[Section 1.5.1.5]{GaRoI}, all DG categories are presentable; therefore DG categories are closed under totalizations by \cite[Corollary 5.5.2.4]{LuHTT}.} in $\C$ and $L$ preserves this totalization in the sense that the canonical map $L(\text{Tot}(C^{\bullet})) \to \text{Tot}(L(C^{\bullet}))$ is an equivalence.
\end{itemize}
\end{Theorem}

\subsubsection{A Comonadicity Condition} In this section, we state and prove a condition for the comonadicity of functors, see \cref{ComonadicityCorollary}. The results of this subsection are modifications of ideas contained in the proof of \cite[Proposition 3.7.1]{Ras3}.

\begin{Proposition}\label{ComonadicityMiracleFact}
If $L: \C \to \D$ is any functor of DG categories equipped with $t$-structures such that the $t$-structures on $\C$ and $\D$ are right-complete and $L$ is $t$-exact, then the induced functor $L: \C^{\geq 0} \to \D^{\geq 0}$ commutes with arbitrary totalizations.
\end{Proposition}

We first begin with a standard lemma on cosimplicial sets, whose proof can be found, for example, in the third paragraph of the proof of \cite[Proposition 3.7.1]{Ras3}:

\begin{Lemma}\label{TruncationIdentity}
For any DG category $\C$ equipped with a right-complete $t$-structure and any cosimplicial object $\F^{\bullet}$ of $\C$ such that $\F^i \in \C^{\geq 0}$ for all $i$, the totalization $\text{Tot}(\F^{\bullet})$ exists and we have the identity
\raggedbottom
\begin{equation*}
\tau^{\leq n}(\text{Tot}(\F^{\bullet})) \simeq \tau^{\leq n}(\text{Tot}^{\leq n + 1}(\F^{\bullet}))
\end{equation*}

\noindent where $\text{Tot}^{\leq n + 1}(\F^{\bullet})$ denotes the partial totalization, i.e. the limit over $\Delta_{\leq n + 1}$. 
\end{Lemma}

\begin{proof}[Proof of \cref{ComonadicityMiracleFact}]
We have
\newcommand{\Fbullet}{\F^{\bullet}}
\[L(\text{Tot}(\F^{\bullet})) \xleftarrow{\sim} L(\text{colim}_n(\tau^{\leq n}(\text{Tot}(\Fbullet)))) \]\[\simeq \text{colim}_nL(\tau^{\leq n}(\text{Tot}(\Fbullet))) \simeq \text{colim}_nL(\tau^{\leq n}(\text{Tot}^{\leq n + 1}(\Fbullet)))\]

\noindent where the first step uses the right-completeness of the $t$-structure of $\D$, the second uses the fact that all functors of DG categories are continuous, and the third uses \cref{TruncationIdentity}. We may continue this chain of equivalences to obtain
\raggedbottom
\[L(\text{Tot}(\F^{\bullet})) \simeq \text{colim}_n(\tau^{\leq n}(\text{Tot}^{\leq n + 1}(L\Fbullet))) \simeq \text{colim}_n(\tau^{\leq n}(\text{Tot}(L\Fbullet)))\]

\noindent where the first step follows from the fact that $L$ is $t$-exact and commutes with finite limits and the second equivalence follows from \cref{TruncationIdentity}. In particular, by the right-completeness of the $t$-structure of $\D$, we see that $L$ preserves these totalizations, as desired. 
\end{proof}

\begin{Corollary}\label{ComonadicityCorollary}
Let $\C, \D$ be DG categories equipped $t$-structures for which the $t$-structure on $\C$ and $\D$ are right-complete, and assume that $L: \C \to \D$ is a $t$-exact functor which admits a right adjoint $R$. Then if $L$ is conservative on $\C^{\heartsuit}$, the restricted functors $\C^{\geq 0} \to \D^{\geq 0}$ and $\C^+ \to \D^+$ are comonadic. 
\end{Corollary}

\begin{proof}
By induction on cohomological amplitude, the $t$-exactness of $L$ implies that $L$ sends no nonzero object of finite cohomological amplitude to zero. By right-completeness of the $t$-structure on $\C$ and the fact that $L$ commutes with colimits, we see that $L$ sends no nonzero object of $C^{+}$ to zero, and therefore both restrictions $\C^{\geq 0} \to \D^{\geq 0}$ and $\C^+ \to \D^+$ of $L$ are conservative.

Next, notice that if $C^{\bullet}$ is an $L$-split cosimplicial object of $\C^{\geq 0}$ then by \cref{TruncationIdentity} its totalization exists in $\C$ and by \cref{ComonadicityMiracleFact} $L$ commutes with these totalizations. Thus by \cref{BarrBeck} we see that $\C^{\geq 0} \to \D^{\geq 0}$ is comonadic, and an identical argument shows $\C^{\geq -i} \to \D^{\geq -i}$ is comonadic for any integer $i$. This, in turn, implies that the functor $\C^+ \to LR\text{-comod}(\D^+)$ is fully faithful and essentially surjective, and so it is an equivalence of categories as desired.
\end{proof}

\newcommand{\quotientMapFromLTdToQuotientByCharacterLattice}{q}
\section{Proofs of \cref{Mellin Transform for biWhittaker Sheaves} and \cref{N Averaging Fully Faithful on Whittaker Subcategory}}\label{ProofOfFirstTheorem} 
In this section, we prove \cref{Mellin Transform for biWhittaker Sheaves} and \cref{N Averaging Fully Faithful on Whittaker Subcategory}. 

\subsection{Reminder on Descent to the Coarse Quotient}
In this section, we survey the results we will need, primarily taken from \cite{GannonDescentToTheCoarseQuotientForPseudoreflectionAndAffineWeylGroups}, on the \textit{coarse quotient} $\LTd\sslash\Wext$, a prestack defined precisely in \cite[Section 4.2.1]{GannonDescentToTheCoarseQuotientForPseudoreflectionAndAffineWeylGroups}. By construction, this space admits a quotient map $\quotientmapforcoarsequotient: \LTd \to \LTd\sslash \Wext$, and pullback by this map induces a fully faithful functor \[\quotientmapforcoarsequotient^!: \IndCoh(\LTd\sslash \Wext) \xrightarrow{} \IndCoh(\LTd)^{\Wext}\] by \cite[Theorem 4.18]{GannonDescentToTheCoarseQuotientForPseudoreflectionAndAffineWeylGroups}. If an object of $\IndCoh(\LTd)^{\Wext}$ lies in the essential image of this functor, we say it \textit{descends to the coarse quotient}. In \cite[Theorem 4.23]{GannonDescentToTheCoarseQuotientForPseudoreflectionAndAffineWeylGroups}, we give some conditions on a given sheaf descending to the coarse quotient; we record the ones we will use from \emph{loc.\ cit.} as well give a new condition involving the Mellin transform here:

\begin{Proposition}\label{New Various Conditions for Wext Equivariant Sheaf to Satisfy Coxteter Descent}
A sheaf $\mathcal{F} \in \IndCoh(\LTd)^{\Wext}$ descends to the coarse quotient $\LTd\sslash \Wext$ if and only if it satisfies one of the following equivalent conditions:
\begin{enumerate}
    \item For every field-valued point $x \in \LTd(K)$, the canonical $\Waff_x$-representation on $\overline{x}^!(\text{oblv}^{\Wext}_{\Waff_x}(\F))$ is trivial. 
    
        \item The object $\text{oblv}^{\Wext}_{\langle r \rangle}(\F) \in \IndCoh(\LTd)^{\langle r \rangle}$ descends to the coarse quotient $\LTd\sslash \langle r \rangle$ for every simple reflection $r \in W$.
        
         \item For every simple coroot $\gamma$ with associated simple reflection $r$ of the (finite) Weyl group $W$ and associated closed subgroup scheme $\mathbb{G}_m^{\gamma} \xhookrightarrow{} T$, the functor 
    \raggedbottom
    \[\D(T)^W \xrightarrow{\text{oblv}} \D(T)^{\langle r \rangle} \xrightarrow{\text{Av}_*^{\mathbb{G}_m^{\gamma}}} \D(T/\mathbb{G}_m^{\gamma})^{\langle r \rangle} \simeq \D(T/\mathbb{G}_m^{\gamma}) \otimes \text{Rep}(\langle r \rangle)\]
    
    \noindent maps $\FourierMukai(\F)$ into the subcategory $\D(T/\mathbb{G}_m^{\gamma}) \simeq \D(T/\mathbb{G}_m^{\gamma}) \otimes \text{Vect}_{\text{triv}} \xhookrightarrow{} \D(T/\mathbb{G}_m^{\gamma}) \otimes \text{Rep}(\langle r\rangle)$, where $\mathbb{G}_m^{\gamma}$ is the rank 1 subgroup scheme of $T$ associated to $\gamma$, and $\text{Vect}_{\text{triv}}$ is the full subcategory generated by the trivial representation of the order two group $\langle r \rangle$.
    \end{enumerate}
\end{Proposition} 

\begin{proof}
The equivalence of an $\F$ descending to the coarse quotient, (1), and (2) are given in \cite{GannonDescentToTheCoarseQuotientForPseudoreflectionAndAffineWeylGroups}. We now show that the full subcategories given by (2) and (3) of \cref{New Various Conditions for Wext Equivariant Sheaf to Satisfy Coxteter Descent} are equivalent. To this end, fix a simple reflection $s$, and note that if $F$ denotes the composite functor 
\raggedbottom
\[\D(T)^{\langle r \rangle} \xrightarrow{\text{Av}_*^{\mathbb{G}_m^{\gamma}}} \D(T/\mathbb{G}_m^{\gamma})^{\langle r \rangle} \simeq \D(T/\mathbb{G}_m^{\gamma}) \otimes \text{Rep}(\langle r \rangle)\]

\noindent then for any $\F \in \D(T)^W, F(\F)$ lies in the subcategory $\D(T/\mathbb{G}_m^{\gamma}) \simeq \D(T/\mathbb{G}_m^{\gamma}) \otimes \text{Vect}_{\text{triv}}$ if and only if the object $\text{Av}_*^{T/\mathbb{G}_m^{\gamma}, w}F(\F)$ lies in the full subcategory $\D(T/\mathbb{G}_m^{\gamma})^{T/\mathbb{G}_m^{\gamma}, w} \simeq \D(T/\mathbb{G}_m^{\gamma})^{T/\mathbb{G}_m^{\gamma}, w} \otimes \text{Vect}_{\text{triv}}$, by the conservativity of weak averaging, see \cite{Gai1Aff}. Furthermore, if we assume $r$ has associated coroot $\gamma$ such that $s$ reflects across the hyperplane $\LTd_{\gamma = 0} := \{\gamma = 0\} \xhookrightarrow{} \LTd$, the Mellin transform allows us to identify the functor $\text{Av}_*^{T/\mathbb{G}_m^{\gamma}, w}F$ with the composite:
\raggedbottom
\[\IndCoh(\LTd/\characterlatticeforT)^{\langle r \rangle} \to \IndCoh(\LTd_{\gamma \in \mathbb{Z}}/X')^{\langle r \rangle}\] \[\simeq \IndCoh(\LTd_{\gamma \in \mathbb{Z}}/X') \otimes \text{Rep}(\langle r \rangle) \xrightarrow{\text{oblv}^{X'} \otimes \text{ id}} \IndCoh(\LTd_{\gamma \in \mathbb{Z}}) \otimes \text{Rep}(\langle r \rangle)\]

\noindent where $X'$ is the lattice of weights generated by the fundamental weights distinct from the fundamental weight associated to $\gamma$. Therefore, any sheaf satisfying $(2)$ immediately satisfies $(3)$, and any sheaf satisfying $(3)$ also satisfies $(2)$.
\end{proof}

As in \cite[Section 4.2.4]{GannonDescentToTheCoarseQuotientForPseudoreflectionAndAffineWeylGroups}, we can define a $t$-structure on $\IndCoh(\LTd\sslash\Wext)$ by declaring $\IndCoh(\LTd\sslash\Wext)^{\leq 0}$ to be the full ordinary $\infty$-subcategory closed under colimits and containing $\quotientmapforcoarsequotient_*^{\IndCoh}(\mathcal{O}_{\LTd})$. We also can similarly define a $t$-structure on $\IndCoh(\LTd/\Wext)$ (respectively, $\IndCoh(\LTd/\characterlatticeforT)$) by declaring $\IndCoh(\LTd/\Wext)^{\leq 0}$ to be the full ordinary $\infty$-subcategory closed under colimits and containing the respective IndCoh pushforward given by the quotient map of the structure sheaf $\mathcal{O}_{\LTd}$. Let $\Gamma_{\Wext}$ denote the balanced product $\Wext \mathop{\times}\limits^{\Waff} \Gamma_{\Waff}$, where $\Gamma_{\Waff}$ denotes the union of graphs of the affine Weyl group. We recall some results of \cite{GannonDescentToTheCoarseQuotientForPseudoreflectionAndAffineWeylGroups} which will be used below. 

\begin{Theorem}\label{Base Change and t-Exactness for Coarse Quotient Functors}With the above notation, we have \begin{enumerate}
    \item The commutative diagram
\begin{equation*}
  \xymatrix@R+2em@C+2em{
   \Gamma_{\Wext} \ar[r]^{s} \ar[d]^{t}& \LTd \ar[d]^{\quotientmapforcoarsequotient} \\
  \LTd \ar[r]^{\quotientmapforcoarsequotient} & \LTd\sslash \Wext
  }
 \end{equation*}
 
 \noindent is Cartesian, where $s$, respectively $t$, denote the source and target maps, and $\quotientmapforcoarsequotient$ denotes the quotient map. 
 \item The maps $s$, $t$, and $\quotientmapforcoarsequotient$ are ind-finite flat and the functors $s^!, t^!, \quotientmapforcoarsequotient^!$ admit left adjoints, denoted $s_{\ast}^{\IndCoh}, t_*^{\IndCoh}$, and $\quotientmapforcoarsequotient_*^{\IndCoh}$ respectively. 
 \item The functors $s_{\ast}^{\IndCoh}, s^{!}, t_{\ast}^{\IndCoh}, t^{!}, \quotientmapforcoarsequotient_{\ast}^{\IndCoh}, \quotientmapforcoarsequotient^{!}$ are all $t$-exact. 
 \item There is a canonical base change isomorphism $t_*s^! \xrightarrow{\sim} \quotientmapforcoarsequotient^!\quotientmapforcoarsequotient_*^{\IndCoh}$.
 \end{enumerate}
\end{Theorem}

\begin{Proposition}
Analogous results to \cref{Base Change and t-Exactness for Coarse Quotient Functors} also hold if we replace $\LTd\sslash \Waff$ with $\LTd/\Waff$ or $\LTd/\characterlatticeforT$. For example, 
\begin{equation*}
  \xymatrix@R+2em@C+2em{
   \Wext \times \LTd \ar[r]^{\text{act}} \ar[d]^{\text{proj}}& \LTd \ar[d]^{\quotientmaptostackquotient} \\
  \LTd \ar[r]^{\quotientmaptostackquotient} & \LTd/\Wext
  }
 \end{equation*}
\noindent is Cartesian and base change holds. Moreover, the pullback and pushforward functors given by this diagram are $t$-exact. 
\end{Proposition}

\subsection{Reminder on Nondegenerate $G$-categories} Recall that in \cite[Section 3.2.1]{GannonNew} we defined a certain subcategory $\D(N\backslash G)_{\mathrm{nondeg}}$ of $\D(N\backslash G)$ closed under the left $T$ and right $G$-action and which admits adjoint functors \[J^!:\D(N\backslash G) \leftrightarrow{} \D(N\backslash G)_{\mathrm{nondeg}}: J_*\] for which $J_*$ is fully faithful. For the convenience of the reader, we now record the results of \cite[Section 3]{GannonNew} on these categories which will be used below.

\begin{Theorem}\label{Summary of Companion Paper}With the above notation, we have the following: \begin{enumerate}
    \item The restriction of the $N$-averaging functor \[\AvN: \D(G) \to \D(G)^N \simeq \D(N\backslash G)\] (the right adjoint to the forgetful functor) to the subcategory $\D(G)^{N^-, \psi}$ naturally factors through $\D(N\backslash G)_{\mathrm{nondeg}}$.
    \item The action of $T \times G$ on $\D(N\backslash G)_{\mathrm{nondeg}}$ upgrades to an action of $(T \rtimes W) \times G$, and, for any root $\alpha$, the induced action of the order two group $\langle r_{\alpha} \rangle$ on $\D(N\backslash G)_{\mathrm{nondeg}}^{\mathbb{G}_m^{\alpha}}$ is trivial,  where $\mathbb{G}_m^{\alpha}$ is the image of $\alpha^{\vee}$ . 
    \item The canonical map $\D(G/_{\psi}N)^N \xrightarrow{} \D(G/_{\psi}N^-)^N_{\mathrm{nondeg}}$ is an equivalence, and one can upgrade the equivalence \cref{SupportOfWhittakerSheaves} to an equivalence of categories with an action of $T \rtimes W$.
    \item There is a right-complete $t$-structure on $\D(N\backslash G)_{\mathrm{nondeg}}$ for which $J^!$ is $t$-exact. Similarly, one can equip $\D(N\backslash G/N)_{\mathrm{nondeg}}$ and $\D(N\backslash G/N)^{T \times T, w}_{\mathrm{nondeg}}$ with right-complete $t$-structures.
    \item The category $\D(N\backslash G/N)_{\mathrm{nondeg}}$, respectively $\D(N\backslash G/N)_{\mathrm{nondeg}}^{T \times T, w}$, is generated by the objects in the $W$-orbit, respectively $\Wext$-orbit, of the monoidal unit.
\end{enumerate}
\end{Theorem}

\begin{proof}
    The first statement is \cite[Proposition 3.14]{GannonNew}, and the second is given by formally combining \cite[Corollary 3.37, Corollary 3.45]{GannonNew}. Moreover, the third is given by \cite[Proposition 3.23]{GannonNew}, the fourth is given by \cite[Proposition 3.18]{GannonNew}, and the final statement is a combination of \cite[Corollary 3.41]{GannonNew} and \cite[Corollary 3.42]{GannonNew}.
\end{proof}

We also recall a key tool used in the proof of \cref{Summary of Companion Paper}: namely, the existence of a partial left adjoint to $\AvN$, which follows directly from combining \cite[Theorem 2.43]{GannonNew}, \cite[Corollary 2.44]{GannonNew} and \cite[Theorem 1.5.4]{Gin}:

\begin{Theorem}\label{T Exactness of Shifted Psi Averaging}
If $\C$ is a DG category with a $G$-action, then the functor $\AvN: \C^{N^{-}, \psi} \to \C^N$ admits a left adjoint $\Avpsi[2\dim(N)]$. If, moreover, $\C$ admits a $t$-structure compatible with the $G$-action, then the cohomological shift of these functors $\Avpsishifted: \C^N\to \C^{N^-, \psi}$ and $\AvNshifted: \C^{N^-, \psi} \to \C^{N}$ are $t$-exact. \hfill \qedsymbol
\end{Theorem}

For any category $\C$ with a $G$-action, we use the notatioal shorthand $\mathsf{A}_* := \AvNshifted$ and $\mathsf{A}^{\psi}_! := \Avpsishifted$.

Finally, we recall following result of the companion paper \cite{GannonNew}, which follows formally from combining \cite[Theorem 1.6]{GannonNew} and \cite[Corollary 2.19]{GannonNew}:

\begin{Theorem}\label{Field Valued Point Theorem}
If $F: \C \to \D$ is a functor of $G$-categories which admits a (continuous) adjoint, then this functor is fully faithful (respectively, an equivalence) if and only if the induced functor $F: \C^{B, \mathcal{L}_{[\lambda]}} \to \D^{B, \mathcal{L}_{[\lambda]}}$ is fully faithful (respectively, an equivalence) for all field-valued points $\lambda$ of $\LTd$. Moreover, an object $\F \in \C$ is zero if and only if $\mathrm{Av}^{B_{\lambda}}_*(\F)$ is zero for all field-valued points $\lambda$. \hfill \qedsymbol
\end{Theorem}


\newcommand{\AvNLifted}{\widetilde{\mathrm{Av}}_*^N}

\subsection{Proof of \cref{N Averaging Fully Faithful on Whittaker Subcategory}}
In this subsection, we prove \cref{N Averaging Fully Faithful on Whittaker Subcategory}. Let $\C$ be a $G$-category. Using the fact that invariants and coinvariants agree (see \cite[Corollary 3.1.5]{GaiWhit}) and that tensor products commute with colimits, we may identify $\AvN$ with the functor
\raggedbottom
\[\C^{N^{-}, \psi} \simeq \C_{N^{-}, \psi} \simeq \D(G)^{N^{-}, \psi} \otimes_G \C \xrightarrow{\AvN \otimes_G \text{id}_{\C}} \D(G)^{N}_{\text{nondeg}} \otimes_G \C\] \[\ \simeq \D(G)_{N, \text{nondeg}} \otimes_G \C \simeq \C_{N, \text{nondeg}} \simeq \C^{N}_{\text{nondeg}}\]

\noindent and similarly for the adjoint $\Avpsi$.  We therefore obtain that the construction of the lift of $\AvN$ and its fully faithfulness follows by proving the general universal case: 

\begin{Theorem}\label{UniversalAvNIsFullyFaithful}
The functor of $G$-categories $\AvN: \D(G)^{N^{-}, \psi} \xrightarrow{} \D(G)^N$ lifts to a fully faithful functor of $G$-categories $\AvNLifted: \D(G)^{N^{-}, \psi} \xhookrightarrow{} \D(G)^{N, W}_{\text{nondeg}}$. Moreover, the functor $\AvNshifted$ lifts to a fully faithful $t$-exact functor of $G$-categories $\AvNShiftedAndLifted: \D(G)^{N^-, \psi} \xhookrightarrow{} \D(G)^{N, W}_{\text{nondeg}}$.
\end{Theorem}

\begin{proof}[Proof of \cref{UniversalAvNIsFullyFaithful}]
To construct the lift of $\AvNshifted$, we will construct a lift $\AvNLifted$ of $\AvN$ and set $\AvNShiftedAndLifted := \AvNLifted[\mathrm{dim}(N)]$. As have already recalled that the functor $\AvNshifted$ is $t$-exact in \cref{T Exactness of Shifted Psi Averaging}, the $t$-exactness of $\AvNshifted$ follows since the forgetful functor $\text{oblv}^W$ reflects the $t$-structure. 

To construct $\AvNLifted$, note that the $G$-functor $\AvN$ is given by an integral kernel in $\D(N\backslash G)^{N^-, -\psi}_{\text{nondeg}}$, and, under the equivalences $\D(N\backslash G)^{N^-, -\psi}_{\text{nondeg}} \simeq \D(N\backslash G)^{N^-, -\psi} \simeq \D(T)$ given by \cref{Summary of Companion Paper}(1) and \cref{SupportOfWhittakerSheaves} respectively, this kernel is given by $\delta_1[-\text{dim}(N)] \in \D(T)$ by direct computation or the non-equivariant version of the argument in \cref{MonoidalityOfAvN} below. Furthermore, these equivalences are canonically $W$-equivaraint by \cite[Proposition 5.5.2]{Gin} (see also \cref{Summary of Companion Paper}(3)) and therefore the integral kernel can be canonically equipped with $W$-equivariant structure since $\ast \xhookrightarrow{} T$ is $W$-equivariant. 

We wish to show this lift is fully faithful. By \cref{Field Valued Point Theorem}, it suffices to show that the resulting functor on invariants $\text{Vect} \simeq \D(G/_{\lambda}B)^{N^-, \psi} \to \D(G/_{\lambda}B)^{N, W}_{\text{nondeg}}$ is fully faithful for any field-valued point $\lambda$ of $\LTd$. Since averaging functors respect extension of scalars (which follows immediately from \cite[Lemma 2.49]{GannonNew}) it suffices to assume that $\lambda$ is a $k$-point. By continuity of $\AvN$, we may show the counit of the adjunction is an isomorphism on the one dimensional vector space $k \in \text{Vect}$. By \cite[Proposition 4.6]{GannonNew}(1), this object is a cohomological shift of the direct sum of the indecomposable antidominant injectives $\underline{I}$, or, equivalently by self duality, the direct sum of the antidominant projectives $\underline{P}$. Soergel's endomorphismensatz \cite{Soe1} gives $W \mathop{\times}\limits^{W_{[\lambda]}} C_{\lambda} \xrightarrow{\sim} \uEnd_{\D(G/B_{\lambda})^N_{\text{nondeg}}}(\underline{P})$. Therefore, since $T \rtimes W$ acts on $\D(G/_{\lambda}B)^N_{\text{nondeg}}$ we have that by \cite[Proposition 4.5]{GannonNew} that this equivalence of classical algebras is $W$-equivariant. Therefore we see
\raggedbottom
\[\uEnd_{\D(G/B_{\lambda})^{N,W}_{\text{nondeg}}}(\underline{P}_{\lambda}) \simeq \uEnd_{\D(G/B_{\lambda})^N_{\text{nondeg}}}(\underline{P}_{\lambda})^W \simeq (W \mathop{\times}\limits^{W_{[\lambda]}} C_{\lambda})^W \simeq k\]

\noindent where the second to last equivalence follows since $\underline{P}_{\lambda} \simeq \AvN(k)$ lies in the nondegenerate subcategory, and the last equivalence follows since $C_{\lambda}^W$ can be identified with the regular $W$-representation by Soergel's endomorphismensatz. Therefore, we see our functor is fully faithful.
\end{proof}

For any category $\C$ with a $G$-action, we use the same notation as in the \lq universal case\rq{} of \cref{UniversalAvNIsFullyFaithful} for the functor \[\AvNShiftedAndLifted := \AvNLifted[\mathrm{dim}(N)]: \C^{N^-, \psi} \to \C^{N, W}_{\mathrm{nondeg}}\]  where $\AvNLifted$ is the corresponding lift of $\AvN: \C^{N^-, \psi} \to \C^N_{\mathrm{nondeg}}$ to $\C^{N, W}_{\mathrm{nondeg}}$. If $\C$ is equipped with a $t$-structure compatible with the $G$-action, then $\AvNShiftedAndLifted$ is $t$-exact by the $t$-exactness in the universal case of \cref{UniversalAvNIsFullyFaithful}, completing the proof of \cref{N Averaging Fully Faithful on Whittaker Subcategory}. 

\subsection{Monoidality of Averaging Functor}
In this section, we prove the following proposition: 
\raggedbottom
\begin{Proposition}\label{MonoidalityOfAvN}
    The composite functor \[\Hpsi \xrightarrow{\AvNShiftedAndLifted} \D(N\backslash G/_{-\psi}N^-)^W \simeq \D(T)^W\] is monoidal. 
\end{Proposition}

   \begin{proof}
The functoriality of averaging gives that the following diagram canonically commutes:
\noindent
\begin{equation}\label{N Invariants of End of Whittaker Category is N Averaging to Torus}
  \xymatrix@R+2em@C+2em{
   \uEnd_G(\D(G/_{-\psi}N^-)) \ar[r]^{(-)^N} \ar[d]^{\text{ev}_{\delta_{N^{-}, \psi}}} & \uEnd_{T \rtimes W}(\D(N\backslash G/_{-\psi}N^-)) \ar[d]^{\text{ev}_{\delta_{N\backslash NN^-/N^-}}}\\
   \Hpsi \ar[r]^{\AvNShiftedAndLifted} & \D(N\backslash G/_{-\psi}N^-)^W
  }
 \end{equation}
 
 \noindent where the vertical arrows are the evaluation maps and the top arrow is the functor induced by $W$-equivariance, see \cref{Summary of Companion Paper}(3). Similarly, note if $\jmath^*$ denotes the isomorphism of \cref{SupportOfWhittakerSheaves} (which we recall is $W$-equivariant, again by \cref{Summary of Companion Paper}(3)) then we have a canonical identification of the following diagram 
 
 \raggedbottom
 \begin{equation}\label{Isomorphism of Categories for SupportOfWhittakerSheaves Yields Commuting Evaluation Diagram}
  \xymatrix@R+2em@C+2em{
\uEnd_{T \rtimes W}(\D(N\backslash G/_{-\psi}N^-))\ar[d]^{\text{ev}_{\AvN(\delta_{N^{-}, \psi})}} \ar[r]^{\textcolor{white}{hello}I \circ - \circ I^{-1}} & \uEnd_{T \rtimes W}(\D(T)) \ar[d]^{\text{ev}_{\delta_{1}}}\\
  \D(N\backslash G/_{-\psi}N^-)^W \ar[r]^{\jmath^*} & \D(T)^W
  }
 \end{equation}
 
 \noindent since $I(\AvN(\delta_{N^-, \psi})) \simeq \delta_1 \in \D(T)$. Now, note that the functors given by the top horizontal arrows of \labelcref{N Invariants of End of Whittaker Category is N Averaging to Torus} and \labelcref{Isomorphism of Categories for SupportOfWhittakerSheaves Yields Commuting Evaluation Diagram} are monoidal, the left vertical arrow of \labelcref{N Invariants of End of Whittaker Category is N Averaging to Torus} is a monoidal equivalence, and the right vertical arrow of \labelcref{Isomorphism of Categories for SupportOfWhittakerSheaves Yields Commuting Evaluation Diagram} is a monoidal equivalence. Thus, since these diagrams commute, we see that $I\AvNShiftedAndLifted$ is a composite of monoidal functors. 
\end{proof}

In what follows we lightly abuse notation and use the same symbol $\AvNShiftedAndLifted$ for the composite of functors in \cref{MonoidalityOfAvN}.

\subsection{Proofs of \cref{Mellin Transform for biWhittaker Sheaves} from \cref{N Averaging Fully Faithful on Whittaker Subcategory}}\label{Reduction to Fixed Central Character}
In this subsection, we verify the essential image of our shifted and lifted functor $\AvNShiftedAndLifted$ and complete the proofs of \cref{Mellin Transform for biWhittaker Sheaves} and \cref{Abelian Categorical Mellin Transform for biWhittaker Sheaves}. We first make the following computation on the essential image:

\begin{Proposition}\label{Averaging Essential Image of AvN Has Trivial Representation of Order Two Reflection Group}
Fix some simple root $\alpha$. The composite $\overline{\mathsf{A}}_*^{\alpha}$ of the functors 
\raggedbottom
\[\D(G)^{N^{-}, \psi} \xrightarrow{\AvNShiftedAndLifted} \D(G/N)_{\text{nondeg}}^{\langle r_{\alpha} \rangle} \xrightarrow{\text{Av}_*^{\mathbb{G}_m^{\alpha}}} \D(G/N)_{\text{nondeg}}^{\mathbb{G}_m^{\alpha} \rtimes \langle r_{\alpha} \rangle}  \simeq \D(G/N)_{\text{nondeg}}^{\mathbb{G}_m^{\alpha}} \otimes \text{Rep}(\langle r_{\alpha} \rangle) \]

\noindent where the final equivalence is given by \cref{Summary of Companion Paper}(2), factors through the subcategory labelled by the trivial representation.
\end{Proposition}
\begin{proof}
By \cite{BZGO}, it suffices to show that 
\raggedbottom
\[\overline{\mathsf{A}}_*^{\alpha}(\delta_{N^{-}, \psi}) \in \D(N\backslash G)_{\text{nondeg}}^{(N^{-}, \psi), \mathbb{G}_m^{\alpha}} \otimes \text{Rep}(\langle r_{\alpha} \rangle)\]

\noindent lies in the full ($G$-)subcategory labelled by the trivial representation. However, by direct computation or \cref{MonoidalityOfAvN}, we have that the underlying object of $\AvNShiftedAndLifted(\delta_{N^{-}, \psi})$ in $\D(T)$ can be identified with $\delta_1 \in \D(T)$. Furthermore, by the $W$-equivariance of the equivalence of \cref{SupportOfWhittakerSheaves} (\cite[Proposition 5.5.2]{Gin}) we see that the given $W$-equivariance  on $\text{Av}_*^N(\delta_{N^{-}, \psi})$ can be identified with the $W$-equivariance on $\delta_1$ given by the $W$-equivariant closed embedding $\ast \xhookrightarrow{} T$.  However, for the equivariant sheaf $\delta_1 \in \D(T)^W$, we see that $\text{Av}_*^{\mathbb{G}_m^{\alpha}}(\delta_1)$ acquires a trivial $\langle r_{\alpha} \rangle$-representation. Therefore, the $\langle r_{\alpha} \rangle$ action is trivial and so the same holds for $\overline{\mathsf{A}}_*^{\alpha}(\delta_{N^{-}, \psi}) \simeq \text{Av}_*^{\mathbb{G}_m^{\alpha}}\AvNShiftedAndLifted(\delta_{N^{-}, \psi})$. 
\end{proof}

\begin{Corollary}\label{All Objects of Lifted AvN descend to the coarse quotient}
The functor $\AvNShiftedAndLifted: \Hpsi \to \D(T)^W$ factors through the full subcategory of objects of $\D(T)^W$ which descend to the coarse quotient under the Mellin transform. 
\end{Corollary}

\begin{proof}
By the final point of \cref{New Various Conditions for Wext Equivariant Sheaf to Satisfy Coxteter Descent}, it suffices to show that if $\F \in \Hpsi$ then the canonical $\langle r_{\alpha} \rangle$-representation on $\text{Av}_*^{\mathbb{G}_m^{\alpha}}\AvNShiftedAndLifted(\F)$ is trivial. However, this directly follows from taking $(N^-, \psi)$-invariants of the composite functor of \cref{Averaging Essential Image of AvN Has Trivial Representation of Order Two Reflection Group}. 
\end{proof}

\begin{Lemma}\label{Restriction to Character and Whittaker Averaging on Other Side Commute}
Fix some field-valued point $\lambda$. We have a canonical isomorphism of functors:
\raggedbottom
\begin{equation*}
  \xymatrix@R+2em@C+2em{
  \D(B_{\lambda}\backslash G/N)_{\text{nondeg}}^{T,w} \ar[r]^{\text{Av}_!^{-\psi} \text{oblv}^{N, (T,w)}\text{  }} & \D(B_{\lambda}\backslash G/_{-\psi}N^-) \simeq \text{Vect}\\
  \D(N^-_\psi\backslash G/N)^{T,w} \ar[u]^{\text{Av}_*^{B_{\lambda}}\text{oblv}^{N^{-}, \psi}} \ar[r]^{\text{Av}_!^{-\psi} \text{oblv}^{N, (T,w)}}  & \Hpsi \ar[u]_{\text{Av}_*^{B_{\lambda}}\text{oblv}^{N^{-}, -\psi}}
  }
\end{equation*}
\end{Lemma}

\begin{proof}
In the above diagram, the horizontal arrows are averaging with respect to the right action, and the vertical arrows are averaging with respect to the left action. Therefore since all four functors are maps of $G$-categories, the diagram canonically commutes. 
\end{proof}

In the diagram of \cref{Restriction to Character and Whittaker Averaging on Other Side Commute}, we claim that the associated right adjoints to the horizontal arrows are functors of $\Wext$-categories, where we take the $\Wext$-action to be trivial on the categories of the right side of the diagram. To see this, note that via the Mellin transform, we may identify the right adjoint to the bottom functor as the composite of a $W$-equivariant functor and, via the Mellin transform, the forgetful functor $\text{oblv}^{\characterlatticeforT}$. In particular, the fact that $\Wext$ is placid allows us to apply \cite[Lemma D.4.4]{GaiWhit} to show that the adjoint is also $\Wext$-equivariant and induces a functor on coinvariants. Since invariance is coinvariance for infinite discrete groups (\cite[Proposition 2.7]{GannonNew}) we therefore see:

\begin{Lemma}\label{FakeBasechangeLemmaMellinTransformed}
Fix some field-valued $\lambda \in \LTd$. Then the following diagram canonically commutes:
\raggedbottom
\begin{equation*}
  \xymatrix@R+2em@C+2em{
  \D(B_{\lambda}\backslash G/N)_{\text{nondeg}}^W \ar[r]^{\text{Av}_!^{-\psi}} & \D(B_{\lambda}\backslash G/_{-\psi}N^-) \simeq \text{Vect}\\
  \D(N^-_\psi\backslash G/N)^W \ar[u]^{\text{Av}_*^{B_{\lambda}}\text{oblv}^{N^{-}, \psi}} \ar[r]^{\text{Av}_!^{-\psi}}  & \Hpsi \ar[u]^{\text{Av}_*^{B_{\lambda}}\text{oblv}^{N^{-}, \psi}}
  }
\end{equation*}
\end{Lemma}

\begin{proof}[Proof of \cref{Mellin Transform for biWhittaker Sheaves}] We have seen in \cref{UniversalAvNIsFullyFaithful} that the functor $\AvNShiftedAndLifted$ is fully faithful and is $t$-exact by \cref{T Exactness of Shifted Psi Averaging}. Furthermore, we have shown that this functor factors through the full subcategory of sheaves descending to the coarse quotient in \cref{All Objects of Lifted AvN descend to the coarse quotient}. Therefore it remains to show that the adjoint $\Avpsi$ is conservative on this the subcategory of sheaves descending to the coarse quotient. Let $\F \in \D(T)^W$ be a sheaf which descends to the coarse quotient. By considering the fully faithful embedding of the zero category into the full $G$-subcategory $\D(G/N)$ generated by $\text{oblv}^{N^{-}, \psi}(\F)$, by \cref{Field Valued Point Theorem}, we see that there exists some field-valued point $\lambda$ such that $\text{Av}_*^{B_{\lambda}}\text{oblv}^{N^{-}, \psi}(\F)$ does not vanish. Since averaging functors respect extension of scalars (which follows immediately from \cite[Lemma 2.49]{GannonNew}) it suffices to assume $\lambda$ is a $k$-point. 

Now note that the following diagram commtues
\raggedbottom
\begin{equation*}
  \xymatrix@R+2em@C+2em{
  \D(B_{\lambda}\backslash G)^{N^-, -\psi} \ar[r]^{\sim} & \IndCoh(\ast) \\
  \D(B_{\lambda}\backslash G/N)_{\text{nondeg}}^W \ar[r]^{\tildeV} \ar[u]^{\text{Av}_!^{-\psi}[-\text{dim}(N)]}  & \IndCoh(\text{Spec}(C_{\lambda}) \mathop{\times}\limits^{\Waff_{\lambda}} W)^W \ar[u]^{(\terminalmapfromC_*^{\IndCoh}(-))^W}
  }
\end{equation*}

\noindent since again we may identify the images the left adjoints via the image of $k \in \text{Vect}$. Thus since by assumption $\F \in \D(T)^W$ descends to the coarse quotient, by \cref{New Various Conditions for Wext Equivariant Sheaf to Satisfy Coxteter Descent}(1) that the sheaf $\text{Av}_!^{-\psi}\text{Av}_*^{B_{\lambda}}\text{oblv}^{N^-_{\ell}, \psi}(\F)$ does not vanish. Thus by \cref{FakeBasechangeLemmaMellinTransformed}, we see that $\text{Av}_!^{\psi}(\F)$ does not vanish. 
\end{proof}

\subsection{Proof of \cref{Abelian Categorical Mellin Transform for biWhittaker Sheaves} from \cref{Mellin Transform for biWhittaker Sheaves}}
We now complete the proof of \cref{Abelian Categorical Mellin Transform for biWhittaker Sheaves}. As in \cref{Our Argument Shows a t-Exact Equivalence of IndCoh on Coarse Quotient and Hpsi}, it suffices to show that each functor in \cref{Mellin Transform for biWhittaker Sheaves} is $t$-exact. By \cref{T Exactness of Shifted Psi Averaging}, $\AvNshifted$ is $t$-exact, and $\AvNShiftedAndLifted$ is $t$-exact since $\text{oblv}^W$ reflects the $t$-structure. Furthermore, the $!$-pullback by the quotient $\LTd/\Wext \to \LTd\sslash \Wext$ is $t$-exact by \cite[Proposition 4.18]{GannonDescentToTheCoarseQuotientForPseudoreflectionAndAffineWeylGroups}. Therefore it remains to show the following:

\begin{Proposition}
The shifted Mellin transform $\FourierMukai[d]$ is $t$-exact, where $d := \text{dim}(\LTd)$.
\end{Proposition}

\begin{proof}
As we mention above, this follows directly from our definition of the Mellin transform. The Mellin transform also admits a Fourier-Mukai description and we include an alternate proof using this definition. It is standard that the functors
\raggedbottom
\[\IndCoh(T_{dR}) \xrightarrow{\phi^!} \IndCoh(T) \xrightarrow{\Gamma^{\text{IndCoh}}} \text{Vect}\] 

\noindent correspond, under the associated Fourier-Mukai transformations, to the functors
\raggedbottom
\[\IndCoh(\LTd/\characterlatticeforT) \xrightarrow{\Gamma^{\text{IndCoh}}[-d]/\characterlatticeforT} \IndCoh(\ast/\characterlatticeforT) \xrightarrow{c^!} \text{Vect}\]

\noindent where $c: \ast \to \ast/\characterlatticeforT$ is the quotient map, see the proof of \cite[Th\'eor\`{e}me 6.3.3(ii)]{LaumonTransformationDeFourierGeneralisse}, whose proof also applies to IndCoh in the DG categorical context. 

Let $F$ denote the composite $\Gamma^{\text{IndCoh}}\phi^!$ and let $G$ denote the composite $\Gamma^{\text{IndCoh}}[-d]/\characterlatticeforT \circ c^!$. Then by this observation we see that there is a canonical identification exhibiting that the following diagram commutes:

\begin{equation*}
  \xymatrix@R+2em@C+2em{
   \IndCoh(\LTd/\characterlatticeforT)  \ar[r]^{\textcolor{white!100}{space}G} \ar[d]^{\FourierMukai} & \text{Vect} \ar[d]^{\text{id}} \\
   \D(T) := \IndCoh(T_{dR})\ar[r]^{\textcolor{white!100}{space}F} & \text{Vect}
  }
 \end{equation*}
 
 \noindent since the Fourier-Mukai transform for the trivial group is the identity. Because the functors $F$, $G[-d]$, and $\text{id}$ are $t$-exact, we see that $\FourierMukai[d]$ is $t$-exact as well. \hfill \hspace{3in} \qedsymbol 
\end{proof}

\section{Proof of \cref{Main Monoidal Equivalences}}\label{ProofOfMainTheorem}
\newcommand{\FI}{F_{I}}
\newcommand{\FD}{F_{\mathcal{D}}}
\newcommand{\LI}{t_*^{\IndCoh}}
\newcommand{\LD}{\mathsf{A}^{\psi}_!}
\newcommand{\LIenh}{s_*^{\IndCoh, \text{enh}}}
\newcommand{\LDenh}{\text{Av}_!^{\psi, \text{enh}}}
In this section, we prove \cref{Main Monoidal Equivalences}. We will state a variant of \cref{Main Monoidal Equivalences}, \cref{NewBigTheorem}, in \cref{BigEquivAsCats}. In the remaining subsections, we discuss the proofs of \cref{NewBigTheorem} and \cref{Main Monoidal Equivalences}. 

\begin{Remark}\label{Explanatory Remark}
We highlight two (related) technical issues which arise in the proof of \cref{Main Monoidal Equivalences} which are already visible when $\LTd\sslash\Wext$ is replaced with the affine scheme $\LTd\sslash W$. In this case, an elementary computation (performed in for example \cite[Proposition A.2]{GannonDescentToTheCoarseQuotientForPseudoreflectionAndAffineWeylGroups}) identifies $\LTd \times_{\LTd\sslash W} \LTd$ with the union of graphs of $W$ acting on $\LTd$; in particular, $\LTd \times_{\LTd\sslash W} \LTd$ is a scheme rather than an ind-scheme.

\begin{enumerate}
    \item We observe that the functor \begin{equation*}\label{Source Map Pushforward}s_*^{\IndCoh}: \IndCoh(\LTd \times_{\LTd\sslash \Wext} \LTd) \to \LTd\end{equation*} fails to be conservative, even if $\Wext = \Waff$ so that $\LTd \times_{\LTd\sslash \Waff} \LTd$ is an ind-closed subscheme of $\LTd \times \LTd$. This fact implies that one cannot, for example, directly apply the formalism of Barr-Beck-Lurie to understand the entirety of the domain of $s_*^{\IndCoh}$ from comodules for the natural comonad acting on $\IndCoh(\LTd)$.
    
    In fact, we claim that the IndCoh pushforward functor \[\mathring{s}_*^{\IndCoh}: \IndCoh(\LTd \times_{\LTd\sslash W} \LTd) \to \LTd\] to one of the factors often fails to be conservative in the more elementary case where the affine Weyl group is replaced with the finite Weyl group. (By the functoriality of IndCoh pushforward, this fact implies that $s_*^{\IndCoh}$ is not conservative in general.) Indeed, letting $\Gamma := \LTd \times_{\LTd\sslash W} \LTd$, we see that if $\mathring{s}_*^{\IndCoh}$ were conservative, then the functor \[\Psi_{\Gamma}: \IndCoh(\Gamma) \to \QCoh(\Gamma)\] obtained by ind-extension of the inclusion $\mathrm{Coh}(\Gamma) \subseteq \QCoh(\Gamma)$ is conservative as well, since the pushforward functor $\mathring{s}_*$ on quasicoherent sheaves is conservative, there is an isomorphism $\Psi_{\LTd}s_*^{\IndCoh} \simeq s_*\Psi_{\Gamma}$ by \cite[Proposition 3.1.1]{GaiIndCoh}, and the functor $\Psi_{\LTd}$ is an equivalence by \cite[Lemma 1.1.6]{GaiIndCoh}. However, the functor $\Psi_{\Gamma}$ cannot be conservative as it is adjoint to a fully faithful functor $\Xi_{\Gamma}$ (see \cite[Proposition 1.5.3]{GaiIndCoh}) and so the conservativity of $\Psi_{\Gamma}$ would imply that is an equivalence, but by \cite[Proposition 1.6.4]{GaiIndCoh} this would imply that $\Gamma$ is smooth; this violates the fact that $\Gamma$ is singular if $G$ is not a torus.
    
     On the other hand, the functor $\Psi_{\Gamma}$ is an equivalence when restricted to the eventually coconnective subcategory $\IndCoh(\Gamma)^+$, see \cite[Proposition 1.2.4]{GaiIndCoh}. Using this (or easier variants of our arguments below) one can show that the restriction of $\mathring{s}_*^{\IndCoh}$ to the eventually coconnective subcategory is comonadic. Our strategy in proving \cref{Main Monoidal Equivalences} can be informally summarized as arguing an analogue of this comonadicity on the eventually coconnective subcategories of the categories appearing in \cref{Main Monoidal Equivalences} and then \lq bootstrapping\rq{} from our knowledge of the eventually coconnective subcategory to obtain the entirety of \cref{Main Monoidal Equivalences}.

     \item The fact that the quotient map $\LTd \to \LTd\sslash W$ is faithfully flat implies, by \cite[Theorem 1.3]{BenZviFrancisNadlerMoritaEquivalenceforConvolutionCategories}, that there is a Morita equivalence \[\textcolor{white}{white}\QCoh(\LTd\sslash W)\mathrm{-modcat} \xrightarrow{\sim} \QCoh(\LTd \times_{\LTd\sslash W} \LTd)\mathrm{-modcat}\] given by tensoring with the bimodule $\QCoh(\LTd)$. Moreover, using the fact that we have a symmetric monoidal equivalence \[\Upsilon_{S}: \QCoh(S) \xrightarrow{\sim} \IndCoh(S)\] for any smooth classical scheme $S$ given by tensoring with the dualizing complex (see for example \cite[Corollary 5.7.4]{GaiIndCoh}) we therefore obtain that the category $\IndCoh(\LTd\sslash W)$ and the convolution category $\QCoh(\LTd \times_{\LTd\sslash W} \LTd)$ are Morita equivalent. However, one can show that the analogous functor \[\textcolor{white}{white}\IndCoh(\LTd\sslash W)\mathrm{-modcat} \xhookrightarrow{} \IndCoh(\LTd \times_{\LTd\sslash W} \LTd)\mathrm{-modcat}\] obtained by tensoring with $\QCoh(\LTd) \xrightarrow{\sim} \IndCoh(\LTd)$ is fully faithful but is \textit{not} an equivalence.

     Informally, we view the relationship $\Hpsiliteral$-categories (which are equivalently $G$-categories generated by their Whittaker invariants) and nondegenerate $G$-categories as analogous to the relationship between module categories for the category $\IndCoh(\LTd\sslash W)$ and for the convolution monoidal category $\IndCoh(\LTd \times_{\LTd\sslash W} \LTd)$. The monoidal equivalence $F'$ in \cref{Mellin Transform for biWhittaker Sheaves} in particular implies that there is an equivalence of DG categories of $\IndCoh(\LTd\sslash\Wext)$-module categories and $\Hpsiliteral$-categories; however, neither $\Hpsiliteral$ nor $\IndCoh(\LTd\sslash\Wext)$ are Morita equivalent to the category $\IndCoh(\LTd \times_{\LTd\sslash \Wext} \LTd)$ of ind-coherent sheaves\footnote{As one can show that the natural (right) adjoint to the pullback functor $s^*: \QCoh(\LTd) \xrightarrow{} \QCoh(\LTd \times_{\LTd\sslash \Wext} \LTd)$ is not continuous in general, in contrast with the natural left adjoint to the pullback $s^!$, we do not work with the category $\QCoh(\LTd \times_{\LTd\sslash \Wext} \LTd)$ in what follows.} on the ind-scheme $\LTd \times_{\LTd\sslash \Wext} \LTd$. In fact, we claim (but do not prove here) that one can explicitly construct a nonzero a DG category $\C$ which is nondegenerate but has the property that $\C^{N, \psi} \simeq 0$. 
     
     We also expect that the functor $\Upsilon_{\LTd\sslash \Wext}$ provides a symmetric monoidal equivalence \[\QCoh(\LTd\sslash\Wext) \xrightarrow{\sim} \IndCoh(\LTd\sslash\Wext)\] and that there is a fully faithful functor \[\textcolor{white}{whitewhi}\IndCoh(\LTd\sslash \Wext)\mathrm{-modcat} \xhookrightarrow{} \IndCoh(\LTd \times_{\LTd\sslash\Wext} \LTd)\mathrm{-modcat} \] given by tensoring with $\IndCoh(\LTd)$, but we neither use nor prove these expectations in what follows. 
\end{enumerate}
\end{Remark}

\subsection{\cref{Main Monoidal Equivalences} via Endomorphism Categories}\label{BigEquivAsCats}
After recalling the construction of the evaluation functor in \cref{Evaluation Functor Subsubsection}, in \cref{Diagrams from Convolution Formalism Subsubsection} we will state a variant of \cref{Main Monoidal Equivalences}, \cref{NewBigTheorem}.

\subsubsection{Evaluation Functor}\label{Evaluation Functor Subsubsection}If $\mathbf{A}$ is a monoidal DG category and $\C, \D$ are module DG categories, the definition of internal Hom gives a continuous functor $E: \uHom_{\mathbf{A}}(\C, \D) \otimes_{\mathrm{DGCat}} \C \to \D$; specifically, it is the adjoint to the identity functor in $\End(\uHom_{\mathbf{A}}(\C, \D))$. Given a functor $F: \mathrm{Vect} \to \C$, we define the functor $E_{F(k)}$ as the composite \[\uEnd_{\mathbf{A}}(\C, \D) \simeq \uEnd_{\mathbf{A}}(\C, \D) \otimes_{\mathrm{DGCat}} \mathrm{Vect} \xrightarrow{\mathrm{id} \otimes F} \uEnd_{\mathbf{A}}(\C, \D) \otimes_{\mathrm{DGCat}} \C \xrightarrow{E} \D\] and refer to it as the \textit{evaluation functor} at $F(k)$. 

Observe that, if we are given a functor $F: \mathrm{Vect} \to \C$ as above, the $\mathbf{A}$-module structure upgrades this to a map $\tilde{F}: \mathbf{A} \to \C$, informally given by the formula $\tilde{F}(A) := A \star F(k)$, and that we may equivalently write the evaluation functor at $F(k)$ as the composite \[\uHom_{\mathbf{A}}(\C, \D) \xrightarrow{- \circ \tilde{F}} \uHom_{\mathbf{A}}(\mathbf{A}, \D) \xrightarrow{E_{\mathbf{1}_{\mathbf{A}}}} \D\] where $\mathbf{1}_{\mathbf{A}}$ is the monoidal unit, and that $E_{\mathbf{1}_{\mathbf{A}}}$. It is not difficult to verify:

\begin{Proposition}\label{Adjoints of Pulling Back by Adjoints and Evaluation at Monoidal Unit is Equivalence}With the notation as above, we have: \begin{enumerate}
    \item If $\tilde{F}$ admits a left adjoint $\tilde{F}^L$, respectively right adjoint $\tilde{F}^R$, then $\mathcal{P}_{\tilde{F}} := - \circ \tilde{F}$ admits the right adjoint $\mathcal{P}_{\tilde{F}^L} := - \circ \tilde{F}^L$, respectively the left adjoint $\mathcal{P}_{\tilde{F}^R} := - \circ \tilde{F}^R$.
    \item The functor $E_{\mathbf{1}_{\mathbf{A}}}$ is an equivalence of categories.
\end{enumerate}
\end{Proposition}

\begin{proof}
We prove the first claim in the case that $\tilde{F}$ admits a left adjoint; the claim where $\tilde{F}$ admits a right adjoint is completely symmetric. In this case, the unit map for this adjunction for $F \in \uHom_{\mathbf{A}}(\C, \D)$ is given by applying $F$ to the unit map for the adjunction $(\tilde{F}^L, \tilde{F})$. The second claim follows immediately from the Yoneda lemma and the definition of the internal Hom.
\end{proof}
\subsubsection{Diagrams from Convolution Formalism}\label{Diagrams from Convolution Formalism Subsubsection}
We let \begin{equation}\label{LI Definition}\LI: \IndCohx \to \IndCoh(\LTd)\end{equation} be the pushforward associated to the projection map and denote by \begin{equation}\label{LD Definition}\LD: \CatTwTw \to \IndCoh(\LTd)\end{equation} the composite of the left adjoint to the $t$-exact functor $\AvNshifted$, the equivalence $\jmath^*$ of \cref{SupportOfWhittakerSheaves} on the weak $T$-invariants, and the Mellin transform. 

By construction of the evaluation functors and the convolution formalism, the diagram
\begin{equation*}
\xymatrix@R+0em@C+0em{
I(\LTd \times_{\LTd\sslash\Wext} \LTd) \ar[r]^{\mathrm{act}_I\textcolor{white}{w}} \ar[dr]^{t_*^{\IndCoh}} & \textcolor{white}{w}\uEnd_{I(\LTd\sslash\Wext)}(I(\LTd)) \ar[d]^{E_{\omega_{\LTd}}} & \uEnd_{\Hpsiliteral}(\D(N^-_{\psi}\backslash G/N)^{T, w})  \ar[d]^{E_{\delta}} \ar[l]_{\sim}& \ar[l]_{\mathrm{act}_{\D}} \CatTwTw \ar[dl]^{\LD} \\
& I(\LTd) & \D(N^-_{\psi}\backslash G/N)^{T, w} \ar[l]_{\sim} & 
}
 \end{equation*}
commutes, where for notational shorthand we use the symbol $I$ to denote $\IndCoh$, the functors \[\mathrm{act}_I(\F) := \F \otimes^!_{\Symt} (-)\text{ and } \mathrm{act}_{\D}(\mathcal{G}) := \mathcal{G} \star (-)\] are given by the convolution formalism, and each of the horizontal functors appearing in the top row of this diagram are monoidal. We let $\mathrm{act}_I^R$ and $\mathrm{act}_{\D}^R$ denote the \textit{not necessarily continuous} right adjoint to the respective action functors. The counit maps induce natural transformations \begin{equation}\label{IndCoh Counit Comonad Comparison}t_*^{\IndCoh}t^!  \simeq E_{\omega_{\LTd}}\mathrm{act}_I\mathrm{act}_I^RE_{\omega_{\LTd}}^R \xrightarrow{} E_{\omega_{\LTd}}E_{\omega_{\LTd}}^R\end{equation} and \begin{equation}\label{DMod Counit Comonad Comparison}\LD\mathsf{A}^{\psi, R}_! \simeq E_{\delta}\mathrm{act}_{\D}\mathrm{act}_{\D}^RE_{\delta}^R \xrightarrow{} E_{\delta}E_{\delta}^R\end{equation} of comonads acting on their respective domains. The proof of \cref{Main Monoidal Equivalences} will largely reduce to the proof of the following theorem:
 

\begin{Theorem}\label{NewBigTheorem}\noindent We have the following: \begin{enumerate}
\item The functor $\LI: \IndCohx^{+} \to \IndCoh(\LTd)^+$ is comonadic. 
\item The functor $\LD: \CatTwTwplus \to \IndCoh(\LTd)^{+}$ is comonadic. 
\item The functors $E^{\mathrm{enh}}_{\omega_{\LTd}}$ and $E_{\delta}^{\mathrm{enh}}$ are fully faithful.
\item The natural transformations \labelcref{IndCoh Counit Comonad Comparison} and \labelcref{DMod Counit Comonad Comparison} are equivalences.
\item The functors $\mathrm{act}_I$ and $\mathrm{act}_{\D}$ induce equivalences of categories when restricted to the respective not necessarily cocomplete subcategories of eventually coconnective objects.
\end{enumerate}
\end{Theorem}

We will discuss the proofs of points (1), (2), (3), (4), and (5) of \cref{NewBigTheorem} in \cref{Proof of Comonadicity for IndCoh}, \cref{Comonadicity for D Modules}, \cref{Comonadicity of Evaluation Functor Subsection}, \cref{Identification of Comonads Subsection}, and \cref{Proof of NewBigTheorem 5 and Main Monoidal Equivalences}, respectively. 

\subsection{Proof of Comonadicity of $t_*^{\IndCoh}$}\label{Proof of Comonadicity for IndCoh} We now prove point (1) of \cref{NewBigTheorem}. To do this, we appeal to \cref{ComonadicityCorollary}. Since, by construction (see \cref{Base Change and t-Exactness for Coarse Quotient Functors}), $t_*^{\IndCoh}$ admits a continuous right adjoint $t^!$, it therefore suffices to show that $t_*^{\IndCoh}$ is $t$-exact, that the $t$-structure on $\IndCohx$ is right-complete, and that $t_*^{\IndCoh}$ is conservative on $\IndCohx^{\heartsuit}$. We have already recalled in \cref{Base Change and t-Exactness for Coarse Quotient Functors} that $t_*^{\IndCoh}$ is $t$-exact. We prove that the $t$-structure on $\IndCohx$ is right-complete in \cref{Right Completeness of tStructure on IndCoherent Sheaves on Indscheme}, and show that $t_*^{\IndCoh}$ is conservative on $\IndCohx^{\heartsuit}$ in \cref{Conseravtivity of tstarIndCoh on Heart}.

\subsubsection{Right-Completeness of t-Structure of Ind-Coherent Sheaves on Ind-Scheme}\label{Right Completeness of tStructure on IndCoherent Sheaves on Indscheme}
\newcommand{\SymV}{\text{Sym}(V)}
Let $\mathcal{X}$ denote any ind-scheme. Recall the standard $t$-structure on $\IndCoh(\mathcal{X})$ characterized by the property that $\IndCoh(\mathcal{X})^{\geq 0}$ contains precisely those $\mathcal{F} \in \IndCoh(\mathcal{X})$ for which $i^!(\F) \in \IndCoh(X)^{\geq 0}$ for any closed subscheme $X \xhookrightarrow{i} \mathcal{X}$, and which is compatible with filtered colimits \cite[Chapter 4, Section 1.2]{GaRoI}. 

\begin{Proposition}\label{t-Structure on Ind-Scheme is Right-Complete}
The $t$-structure on $\IndCoh(\mathcal{X})$ is right-complete. 
\end{Proposition}

\begin{proof}
Given some $\F \in \IndCoh(\mathcal{X})$, let $\phi^{\mathcal{F}}: \F \to \text{colim}_n \tau^{\leq n}\F$ denote the canonical map, and let $\mathcal{K}$ denote the fiber of $\phi^{\mathcal{F}}$. Since the $t$-structure on $\IndCoh(\mathcal{X})$ is compatible with filtered colimits, we see that $\tau^{\leq n_0}(\phi^{\mathcal{F}})$ is an equivalence for all $n_0 \in \mathbb{Z}$, and that, in particular, $\mathcal{K} \in \IndCoh(\mathcal{X})^{\geq 0}$.

We first prove this in the case where $\mathcal{X}$ is itself a scheme $X$. In this case, we have a $t$-exact equivalence $\Psi_{X}: \IndCoh(X)^{\geq 0} \xrightarrow{\sim} \QCoh(X)^{\geq 0}$. Because $\Psi_X$ is exact, we see that $\Psi_X(\mathcal{K}) \simeq \text{fib}(\Psi_X(\phi^{\mathcal{F}}))$. Because $\Psi$ is continuous and $t$-exact, we obtain a canonical identification $\Psi_X(\phi^{\mathcal{F}}) \simeq \phi^{\Psi_X(\mathcal{F})}$. However, $\phi^{\Psi_X(\mathcal{F})}$ is an equivalence since the $t$-structure on $\QCoh(X)$ is right-complete \cite[Chapter 3, Corollary 1.5.7]{GaRoI}. Thus $\Psi_X(\mathcal{K}) \simeq 0$, and since $\Psi_X$ is in particular conservative on $\IndCoh(X)^{\geq 0}$, we see that $\mathcal{K} \simeq 0$ in this case.

We now assume the result of \cref{t-Structure on Ind-Scheme is Right-Complete} for schemes. For any closed subscheme $i: X \xhookrightarrow{} \mathcal{X}$, we therefore obtain equivalences
\raggedbottom
\[i^!(\mathcal{K}) \xleftarrow{\phi^{i^{!}(\F)}} \text{colim}_n\tau^{\leq n}i^!(\mathcal{K}) \simeq \text{colim}_n\tau^{\leq n}i^!(\tau^{\leq n}\mathcal{K}) \simeq \text{colim}_n\tau^{\leq n}i^!(0)\]

\noindent where the first map is an equivalence $t$-structure on $\IndCoh(X)$ is right-complete, the second step uses the fact that $i^!$ is right t-exact, and the third equivalence is a direct consequence of the above fact that $\tau^{\leq n}(\phi^{\mathcal{F}})$ is an equivalence for all $n$. 
\end{proof}

\newcommand{\AvNTwr}{\mathsf{A}_{(T_{r}, w)}}
\newcommand{\AvNTwl}{\mathsf{A}_{(T_{\ell}, w)}}
\newcommand{\Avpsil}{\LD}

\subsubsection{Conservativity of $t_*^{\IndCoh}$ on Heart}\label{Conseravtivity of tstarIndCoh on Heart}We now prove that $t_*^{\IndCoh}$ is conservative. Since $t_*^{\IndCoh}$ is in particular exact, it suffices to prove that $t_*^{\IndCoh}(\F)$ is nonzero if $\F \in \IndCoh(\LTd \times_{\LTd\sslash\Wext} \LTd)^+$ is nonzero. By the right-completeness of the $t$-structure on $\IndCoh(\LTd \times_{\LTd\sslash\Wext} \LTd)$, there exists some $i$ such that $H^i(\F)$ is nonzero. Since $t_*^{\IndCoh}$ is $t$-exact by \cref{Base Change and t-Exactness for Coarse Quotient Functors}, it suffices to verify that $t_*^{\IndCoh}$ is conservative when restricted to $\IndCoh(\LTd \times_{\LTd\sslash\Wext} \LTd)^{\heartsuit}$. More generally, we have the following Proposition:

\begin{Proposition}\label{Pushforward from Ind-Affine Scheme is Conservative on Heart}
Assume $q: \indsch \to Y$ is a map from an ind-affine scheme $\indsch$ to an affine scheme $Y$. Then $q_*^{\IndCoh}$ is conservative on $\IndCoh(\indsch)^{\heartsuit}$.
\end{Proposition}

To prove \cref{Pushforward from Ind-Affine Scheme is Conservative on Heart}, we first prove the following:

\begin{Lemma}\label{CanFindSubschemeWithCohomology}
Let $\F \in \IndCoh(\indsch)^{\heartsuit}$. Then there exists some closed subscheme $X := X_{\alpha} \xhookrightarrow{i} \indsch$ such that $H^0(i^!(\F))$ is nonzero.  
\end{Lemma}

\begin{proof}
Pick a nonzero $\mathcal{F} \in \IndCoh(\indsch)^{\heartsuit}$. By \cite[Chapter 3, Corollary 1.2.7]{GaRoII}, there exists some closed subscheme $X_{\alpha}\xhookrightarrow{i} \indsch$ and some $\mathcal{G} \in \text{Coh}(X)^{\heartsuit}$ such that the map $i_{S, *}^{\IndCoh}(\mathcal{G}) \to \mathcal{F}$ is nonzero, so that the space $\Hom_{\IndCoh(\indsch)}(i_{S, *}^{\IndCoh}(\mathcal{G}), \F)$ is a discrete space (as both objects lie in the heart of a $t$-structure) with more than one point. Therefore, by adjunction, the same holds for $\Hom_{\IndCoh(X_{\alpha})}(\mathcal{G}, i^!\F)$. However, since $i_*^{\text{IndCoh}}$ is $t$-exact, its right adjoint is left $t$-exact, and so we see that this implies that there exists a nonzero map $\mathcal{G} \to \tau^{\leq 0}i^!\mathcal{F} \simeq H^0(i^!\mathcal{F})$ which obviously implies our claim.
\end{proof}

\begin{proof}[Proof of \cref{Pushforward from Ind-Affine Scheme is Conservative on Heart}]
Pick $\F \in \IndCoh(\indsch)^{\heartsuit}$. Note that, by \cref{CanFindSubschemeWithCohomology}, there exists some closed subscheme $i:X \xhookrightarrow{} \indsch$ for which $H^0(i^!(\F))$ is nonzero, and, by ind-affineness, we may assume $X$ is affine. Let $\pi$ denote the composite $q \circ i: X \to Y$.  Then we have that $H^0(\pi_*^{\IndCoh}i^!(\F)) \simeq H^0(q_*^{\IndCoh}i_{*}^{\IndCoh}i^!(\F))$ is a subobject of $H^0(q_*^{\IndCoh}(\F))$, as $q_{*}^{\IndCoh}$ is ind-affine and thus is $t$-exact by \cite[Chapter 3, Lemma 1.4.9]{GaRoII}. However, we see that $\pi_*^{\IndCoh}$ is conservative (it is the pushforward of an affine morphism) and so $H^0(\pi_*^{\IndCoh}i^!(\F))$ is nonzero, and therefore so too is $H^0(q_*^{\IndCoh}(\F))$.  
\end{proof}
We now record a consequence of \cref{Pushforward from Ind-Affine Scheme is Conservative on Heart} for later use.

\begin{Corollary}\label{Conervativity of Pushforward by QuotientMapforCoarseQuotient}
The functor \[\quotientmapforcoarsequotient_*^{\IndCoh}: \IndCoh(\LTd) \to \IndCoh(\LTd\sslash\Wext)\] is conservative.
\end{Corollary}

Note, this is in contrast to the fact that, as we argued in \cref{Explanatory Remark}(1), the functor $\LI$ in \labelcref{LI Definition} is not conservative unless we restrict to the eventually coconnective subcategory.
\begin{proof}
By base change (\cref{Base Change and t-Exactness for Coarse Quotient Functors}), it suffices to show that $t_*^{\IndCoh}s^!$ is conservative on $\IndCoh(\LTd)$. For a nonzero $\F \in \IndCoh(\LTd)$, there exists some $i$ for which $H^i(\F)$ is nonzero. Since, by \cref{Base Change and t-Exactness for Coarse Quotient Functors}, $t_*^{\IndCoh}s^!$ is $t$-exact, we see that $H^i(t_*^{\IndCoh}s^!\F) \cong t_*^{\IndCoh}s^!H^i(\F)$ and so we may assume $\F \in \IndCoh(\LTd)^{\heartsuit}$. For such an $\F$, we have $s^!(\F) \in \IndCoh(\LTd \times_{\LTd\sslash\Wext} \LTd)^{\heartsuit}$ by \cref{Base Change and t-Exactness for Coarse Quotient Functors} and is nonzero, since for example $\Delta^!s^!(\F) \simeq \F$ for $\Delta$ the diagonal map. Thus by \cref{Pushforward from Ind-Affine Scheme is Conservative on Heart} we see that $t_*^{\IndCoh}s^!(\F)$ is nonzero, as required.  
\end{proof}
\begin{Remark} The structure of the above comonad was previously known on the full subcategory of $B$-bimonodromic objects of the category  $\D(N\backslash G/N)_{\text{nondeg}}^{\heartsuit}$, see \cite[Section 5.1]{Bez2}. 
\end{Remark}

\subsection{Proof of Comonadicity of $\LD$ on Heart}\label{Comonadicity for D Modules}

We have recalled the right-completeness of the $t$-structure on $\D(N\backslash G/N)^{T \times T, w}_{\mathrm{nondeg}}$ in \cref{Summary of Companion Paper}(4). Moreover, by \cref{T Exactness of Shifted Psi Averaging}, the functor $\LD$ is $t$-exact and, by \cref{T Exactness of Shifted Psi Averaging}, admits a continuous right adjoint. The functor $\Avpsi$ is conservative on the eventually coconnective subcategory by construction of $\D(N\backslash G/N)_{\mathrm{nondeg}}$, and thus so too is its cohomological shift $\LD$. Therefore, \cref{NewBigTheorem}(2) follows immediately from \cref{ComonadicityCorollary}. 

\subsection{Fully Faithfulness of Enhanced Evaluation Functors}\label{Comonadicity of Evaluation Functor Subsection} We now prove \cref{NewBigTheorem}(3). To do this, we will first recall some techniques on computing limits of DG categories in \cref{Reminders on Lax Limits Subsubsection} and \cref{EffectiveLimitSection}, and then prove our desired comonadicity in \cref{EndIsComonadicOverIndCohtxtSection}. 
\subsubsection{Reminders on Lax Limits}\label{Reminders on Lax Limits Subsubsection}
We summarize the following useful proposition on computing limits in a limit of categories in DGCat, see \cite[Section 4.1]{AriGai} for more information. Assume $I \to \text{DGCat}$ is a diagram of categories for some index $\infty$-category $I$, which we denote $i \mapsto \C_i$, and let $\text{lim } \C_i$ denote a limit category. We recall the notion of a \textit{lax-limit category} $\text{lax-lim }  \C_i$, which is defined using co-Cartesian fibrations and, in particular, whose objects consist of objects $\F_i \in \C_i$ for each $i \in I$ and a map \begin{equation}\label{Map for Limit of Categories}\Phi_{\alpha}(\F_{i_1}) \to \F_{i_2}\end{equation} for every map $i_1 \xrightarrow{\alpha} i_2$ in $I$.

\begin{Proposition}\label{Condition for Lax Limit to Be In Limit Category} Assume $I \to \text{DGCat}, i \mapsto \C_i$ is defined as above. \begin{enumerate}
\item \cite[Section 4.1.1]{AriGai} There is a natural, fully faithful functor $\text{lim }\C_i \xhookrightarrow{} \text{lax-lim }\C_i$, and an object is in the essential image if and only if the associated maps $\Phi_{\alpha}(\F_{i_1}) \to \F_{i_2}$ are equivalences for all $\alpha$.
\item \cite[Section 4.1.8]{AriGai} For each $i \in I$, the natural evaluation functor $\text{ev}_i: \text{lax-lim }\C_i \to \C_i$ admits a left adjoint, and in particular commutes with limits.
\end{enumerate}
\end{Proposition}

\begin{Corollary}\label{LimitIsTermwise}
Assume we are given a diagram $J \to \text{lim }_i\C_i$, which we write $j \mapsto \F_{j, i} \in \C_i$, such that for each $j$ and for each map $i_1 \xrightarrow{\alpha} i_2$ in $I$, the corresponding map $\Phi_{\alpha}(\F_{j, i_1}) \to \F_{j, i_2}$ is an equivalence. Then the corresponding limit is computed termwise. 
\end{Corollary}

\begin{proof}
The condition that each corresponding map $\Phi_{\alpha}(\F_{j, i_1}) \to \F_{j, i_2}$ is an equivalence implies that the limit over our $J$-shaped diagram, computed in the category $\text{lax-lim }\C_i$, lies in the category $\text{lim }\C_i$. Since the evaluation functor is a right adjoint, it commutes with limits, thus giving our claim. 
\end{proof}

\subsubsection{Nilpotent Towers and Effective Limits}\label{EffectiveLimitSection}
In this section, we recall the DG-analogue of ideas of Akhil Mathew (see, for example, \cite[Subsection 2.3]{Mat}) which will be used later. For this subsection, fix two DG categories $\C, \D$. 

\begin{Definition}
Assume we are given a \textit{tower} in $\mathcal{C}$, or, equivalently, a sequence $... \to \F^2 \to \F^1 \to F^0$ in $\mathcal{C}$. We say this tower is \textit{weakly nilpotent} if for all $n \in \mathbb{N}^{\geq 0}$ there exists an $N$ such that for all $m \geq N$, the natural map $\F^{m + n} \to \F^n$ is nullhomotopic. 
\end{Definition}
    
\begin{Definition}\label{EffectiveDefinition} Let $\C$ be some DG category or, more generally, any stable $\infty$-category, and fix some $\F \in \C$. 
\begin{enumerate}
    \item Let $\F_{\bullet} := (... \to \F_1 \to \F_0)$ be a tower in $\C$, and let $\underline{\F}$ denote the constant tower. We say the map of towers $\underline{\F} \to \F_{\bullet}$ forms an \textit{effective limit} (or, more informally, the maps $\F \to \F_{\bullet}$ \textit{form an effective limit}) if the tower $n \mapsto \text{cofib}(\F \to \F_n)$ is weakly nilpotent. 
    \item Let $S^{\bullet}$ denote some cosimplicial object of a category $\C$ and temporarily denote by $\F^{\bullet}$ the constant cosimplicial object. We say the map of cosimplicial objects $\F^{\bullet} \to S^{\bullet}$ (or, more informally, the maps $\F \to S^{\bullet}$) \textit{form an effective limit} if the maps $\F \to \text{Tot}^{\leq n}(S^{\bullet})$ form an effective limit. 
\end{enumerate}

\end{Definition}

\begin{Remark}\label{EffectiveRemark}
By definition, a tower in $\C$ is an object of the $(\infty, 1)$-category of functors Fun$(\mathbb{Z}^{op}_{\geq 0}, \C)$. Since colimits in functor categories are computed termwise, the cokernel of a map of towers is the tower of cokernels. Note also that if the tower $n \mapsto \text{cofib}(\F \to \F_n)$ is weakly nilpotent, then its limit is zero. 

We therefore see that, if the tower $n \mapsto \text{cofib}(\F \to \F_n)$ is weakly nilpotent, the canonical map $\F \simeq \text{lim}(\underline{\F}) \to \lim_n\F_n$ is an equivalence, so the term \lq effective limit\rq{} is justified. By abuse of notation, we sometimes say that the maps $\F \to \F^i$ form an effective limit. 
\end{Remark}

We now record a basic property of effective limits, see \cite[Proposition 2.20]{Mat}: 

\begin{Lemma}\label{EffectiveLemma} Let $F: \C \to \D$ is some exact functor of stable $\infty$-categories (which is always satisfied if $F$ is a map in $\DGCatContk$), and let $\F \to \F^i$ be a compatible family of maps as in \cref{EffectiveDefinition}. Then if the maps $\F \to \F^i$ form an effective limit in $\C$, then the maps $F(\F) \to F(\F^i)$ form an effective limit in $\D$.
\end{Lemma}

\begin{proof}
By definition of effective limit, the tower of cofibers given by $C_i := \text{cofib}(\F \to \F_i)$ is weakly nilpotent. By the definition of exactness, $F$ commutes with finite colimits, so that $F(C_i) \simeq \text{cofib}(F(\F) \to F(\F_i))$. Therefore, since $F$ preserves the class of maps which are equivalent to the 0 map, our claim follows, since exact functors preserve the zero object. 
\end{proof}

\subsubsection{Proof of \cref{NewBigTheorem}(3)}\label{EndIsComonadicOverIndCohtxtSection} 
\newcommand{\quotientFromCharacterQuotientToGITByWext}{\overline{\phi}}
The entirety of \cref{EndIsComonadicOverIndCohtxtSection} is devoted to the proof of \cref{NewBigTheorem}(3). We first prove the following Lemma: 

\begin{Lemma}\label{Map from Identity to Monad Forms Effective Limit}
For any $F \in \uEnd(\IndCoh(\LTd))$, the maps $F \to F(\quotientmapforcoarsequotient^!\quotientmapforcoarsequotient_*^{\IndCoh})^{\bullet + 1}$ form an effective limit. 
\end{Lemma}

\begin{proof}
Since the functor $F \circ -$ is exact, by \cref{EffectiveLemma} it suffices to show that the maps $\text{id}_{\IndCoh(\LTd)} \to (\quotientmapforcoarsequotient^!\quotientmapforcoarsequotient_*^{\IndCoh})^{\bullet + 1}$ form an effective limit. Using the identification $\text{ev}_{\omega_{\LTd}}: \uEnd(\IndCoh(\LTd)) \xrightarrow{\sim} \IndCoh(\LTd \times \LTd)$ this claim is equivalent to the claim that the maps $\omega_{\LTd} \to (\quotientmapforcoarsequotient^!\quotientmapforcoarsequotient_*^{\IndCoh})^{\bullet + 1}(\omega_{\LTd})$ form an effective limit.  Let $c_n: \omega_{\LTd} \to \text{Tot}^{\leq n}(\quotientmapforcoarsequotient^!\quotientmapforcoarsequotient_*^{\IndCoh})^{\bullet + 1}(\omega_{\LTd})$ denote the canonical map for each $n$. We claim that for each $n$, we have that $\tau^{\leq n - 1}c_n$ is an equivalence. To see this, note that by the conservativity of $\quotientmapforcoarsequotient_*^{\IndCoh}$ (\cref{Conervativity of Pushforward by QuotientMapforCoarseQuotient}) it suffices to show that $\quotientmapforcoarsequotient_*^{\IndCoh}\tau^{\leq n - 1}c_n$ is an equivalence. The $t$-exactness of $\quotientmapforcoarsequotient_*^{\IndCoh}$ (\cref{Base Change and t-Exactness for Coarse Quotient Functors}) allows us to identify this map with $\tau^{\leq n - 1}\quotientmapforcoarsequotient_*^{\IndCoh}c_n$. Since $\quotientmapforcoarsequotient_*^{\IndCoh}$ is exact, it commutes with finite limits, so we may furthermore identify $\tau^{\leq n - 1}\quotientmapforcoarsequotient_*^{\IndCoh}c_n$ with the map 
\raggedbottom
\[\tau^{\leq n - 1}\quotientmapforcoarsequotient_*^{\IndCoh}\omega_{\LTd} \to \tau^{\leq n - 1}\text{Tot}^{\leq n}\quotientmapforcoarsequotient_*^{\IndCoh}(\quotientmapforcoarsequotient^!\quotientmapforcoarsequotient_*^{\IndCoh})^{\bullet + 1}(\omega_{\LTd})\]

\noindent and by \cref{TruncationIdentity} we may further identify this map with the canonical map
\raggedbottom
\[\tau^{\leq n - 1}\quotientmapforcoarsequotient_*^{\IndCoh}\omega_{\LTd} \to \tau^{\leq n - 1}\text{Tot}(\quotientmapforcoarsequotient_*^{\IndCoh}(\quotientmapforcoarsequotient^!\quotientmapforcoarsequotient_*^{\IndCoh})^{\bullet + 1}(\omega_{\LTd}))\]

\noindent using the fact that $(\quotientmapforcoarsequotient^!\quotientmapforcoarsequotient_*^{\IndCoh})^{j + 1}(\omega_{\LTd})$ lies in the heart for every $j \in \mathbb{Z}^{\geq 0}$, see \cref{Base Change and t-Exactness for Coarse Quotient Functors}. However, this map is an equivalence since the cosimplicial object $\quotientmapforcoarsequotient_*^{\IndCoh}(\quotientmapforcoarsequotient^!\quotientmapforcoarsequotient_*^{\IndCoh})^{\bullet + 1}$ is split by $\quotientmapforcoarsequotient_*^{\IndCoh}$. We therefore see that $\tau^{\leq n - 1}c_n$ is an equivalence. 

We also have that $\text{Tot}^{\leq n}(\quotientmapforcoarsequotient^!\quotientmapforcoarsequotient_*^{\IndCoh})^{\bullet + 1}(\omega_{\LTd})$ is a totalization of objects in the (shifted) heart of a category equivalent to $A$-mod for some classical ring $A$, again using the exactness of \cref{Base Change and t-Exactness for Coarse Quotient Functors}. We thus see see $\text{Tot}^{\leq n}(\quotientmapforcoarsequotient^!\quotientmapforcoarsequotient_*^{\IndCoh})^{\bullet + 1}(\omega_{\LTd})$ lies in cohomological degree $[0, n]$. Therefore, if $K^n$ denotes the cofiber of the map $c_n$, this cofiber is concentrated in degree $n$ since $\tau^{\leq n - 1}c_n$ is an equivalence. In particular, we may choose $N \gg 0$ so that the space $\text{Hom}_{\IndCoh(\LTd \times \LTd)}(K^{N + n}, K^n)$ is connected (by the finite cohomological dimension of the $t$-structure on $\IndCoh(\LTd \times \LTd)$), so the maps from the identity to the tower of partial totalizations of our cosimplicial object form an effective limit by definition. 
\end{proof}
\begin{proof}[Proof of \cref{NewBigTheorem}(3)]
We prove  \cref{NewBigTheorem}(3) for $E_{\omega_{\LTd}}$; the proof that $E_{\delta}$ is fully faithful can be immediately deduced from the above commutative diagram or by very similar arguments. Observe that we have a commutative diagram \begin{equation}\label{Longhand diagram}\xymatrix@R+2em@C+2em{\uEnd_{\IndCoh(\LTd\sslash\Wext)}(I(\LTd)) \ar[r]^{\textcolor{white}{wjote}\mathrm{oblv}} \ar[d]^{\mathrm{pullback}_{\quotientmapforcoarsequotient^!}} \ar[dr]^{E_{\omega_{\LTd}}}& \uEnd(\IndCoh(\LTd))) \ar[d]^{E_{\omega_{\LTd}}}\\
\uHom_{\IndCoh(\LTd\sslash\Wext)}(\IndCoh(\LTd\sslash\Wext), \IndCoh(\LTd)) \ar[r]^{\textcolor{white}{whitewhitewhitewhite}E_{\omega_{\LTd\sslash\Wext}}} & \IndCoh(\LTd)
  }\end{equation} where, as above, $\quotientmapforcoarsequotient: \LTd \to \LTd\sslash\Wext$ is the quotient map. In what follows, it will be convenient to use the notational shorthand \newcommand{\I}{\mathrm{I}}
  \begin{equation}\xymatrix@R+2em@C+2em{\uEnd_{\I(\LTd\sslash\Wext)}(I(\LTd)) \ar[r]^{\mathfrak{o}} \ar[d]^{P} \ar[dr]^E & \uEnd(\I (\LTd))) \ar[d]^{\mathfrak{p}}\\
\uHom_{\I (\LTd\sslash\Wext)}(\I (\LTd\sslash\Wext), \I (\LTd)) \ar[r]^{\textcolor{white}{white}e} & \I (\LTd)
  }\end{equation} for each of the objects and functors in the diagram \labelcref{Longhand diagram}. Thus, in this notation, we wish to show that $E^{\mathrm{enh}}$ is fully faithful.

  Since the functor $e$ is an equivalence by \cref{Adjoints of Pulling Back by Adjoints and Evaluation at Monoidal Unit is Equivalence}(2), to show that $E^{\mathrm{enh}}$ is fully faithful it suffices to show that $P^{\mathrm{enh}}$ is fully faithful. To show that $P^{\mathrm{enh}}$ is fully faithful, it suffices to show that, for every $F \in \uEnd_{\I (\LTd\sslash\Wext)}(I(\LTd))$, that the unit map \begin{equation}\label{Unit map for enhanced adjoint}
      u(F): F \to \text{lim}_{\Delta}(P^RP)^{\bullet + 1}(F)
  \end{equation} for $P^{\mathrm{enh}}$ is an equivalence, where $P^R$ is the functor of pullback by $\quotientmapforcoarsequotient_*^{\IndCoh}$.\footnote{The functor $P^R$ is the right adjoint to $P$ by \cref{Adjoints of Pulling Back by Adjoints and Evaluation at Monoidal Unit is Equivalence}(1), which justifies the notation.}
  
  Fix some $F \in \uEnd_{\I (\LTd\sslash\Wext)}(I(\LTd))$. We first claim that the natural map \begin{equation}\label{Oblv of Comparison Map}\mathfrak{o}(\text{lim}_{\Delta} (P^RP)^{\bullet + 1}(F)) \to \text{lim}_{\Delta} \mathfrak{o}((P^RP)^{\bullet + 1}(F))\end{equation} is an equivalence. To this end, let $\imath$ denote the inclusion \[\imath: \uEnd_{\I(\LTd\sslash\Wext)}(\I(\LTd)) \simeq \text{lim}_{\Delta}\uHom(\I (\LTd\sslash\Wext)^{\bullet} \otimes \I (\LTd), \I (\LTd)) \xhookrightarrow{}\text{lax-lim}_{\Delta}\uHom(\I (\LTd\sslash\Wext)^{\bullet} \otimes \I (\LTd), \I (\LTd))\] of $\uEnd_{\I(\LTd\sslash\Wext)}(\I(\LTd))$ into the lax-limit of categories. We will compute the limit \[\text{lim}_{\Delta}\imath(F(\quotientmapforcoarsequotient^!\quotientmapforcoarsequotient_*^{\I})^{\bullet + 1}) = \text{lim}_{\Delta}\imath((P^RP)^{\bullet + 1}F)\] and show it lies the essential image of $\imath$. The explicit description of a limit of categories given in \cite[Corollary 3.3.3.2]{LuHTT} (see also \cite[Section 1.6.2]{DrinfeldGaitsgoryCompactGeneration}) implies that any object $F' \in \uEnd_{\I(\LTd\sslash\Wext)}(\I(\LTd))$ is given by a collection of $F'_{m}$ for every $m \in \Z^{\geq 0}$ such that $F'_{0} \simeq \mathfrak{o}(F)$ and, for each $m$, $F'_m$ is isomorphic to $\Phi_\alpha$ for some map $\alpha: 0 \to m$. In particular, since the collection of maps \[\mathfrak{o}(F) \to (P^RP)^{\bullet + 1}(\mathfrak{o}(F)) = \mathfrak{o}(F)(\quotientmapforcoarsequotient^!\quotientmapforcoarsequotient_*^{\IndCoh})^{\bullet + 1}\] forms an effective limit by \cref{Map from Identity to Monad Forms Effective Limit}, we deduce that, for each $m$, the collection of maps \[F_m \to (F(\quotientmapforcoarsequotient^!\quotientmapforcoarsequotient_*^{\IndCoh})^{\bullet + 1})_m\] form an effective limit for all $m$ by \cref{EffectiveLemma}. Therefore, for any map $\alpha: m \to m'$ in $\Delta$, we obtain that the collection of maps \[\Phi_{\alpha}(F_m) \to \Phi_{\alpha}((F(\quotientmapforcoarsequotient^!\quotientmapforcoarsequotient_*^{\IndCoh})^{\bullet + 1})_m)\] also forms an effective limit, again by \cref{EffectiveLemma}. By \cref{EffectiveRemark}, then, we deduce that \[\Phi_{\alpha}(F_m) \xrightarrow{\sim} \text{lim}_{\Delta} \Phi_{\alpha}((F(\quotientmapforcoarsequotient^!\quotientmapforcoarsequotient_*^{\IndCoh})^{\bullet + 1})_m) \simeq \text{lim}_{\Delta}(F(\quotientmapforcoarsequotient^!\quotientmapforcoarsequotient_*^{\IndCoh})^{\bullet + 1})_{m'})\] for every $\alpha$. Chasing through the definitions, we see that this map is the map of \labelcref{Map for Limit of Categories}. We deduce that we have a canonical equivalence \[\imath(\mathrm{lim}_{\Delta}(P^RP)^{\bullet + 1}F) \xrightarrow{\sim} \mathrm{lim}_{\Delta}\imath((P^RP)^{\bullet + 1}F)\] by \cref{Condition for Lax Limit to Be In Limit Category}(1). Therefore, by \cref{LimitIsTermwise}, the limit $\mathrm{lim}_{\Delta}((P^RP)^{\bullet + 1}F)$ is computed termwise. Therefore, the map \labelcref{Oblv of Comparison Map} is an isomorphism, as desired.

Now observe that $\mathfrak{p}$ commutes with limits: indeed, we may identify $\mathfrak{p}$ with the composite \[\uEnd(\I(\LTd)) \simeq \I(\LTd \times \LTd) \xrightarrow{p_*^{\IndCoh}} \I(\LTd)\] where $p$ is the projection onto the first factor. Applying the functor $\mathfrak{p}$ to  \labelcref{Oblv of Comparison Map} and using this, we deduce that the canonical map \begin{equation}\label{E of Comparison Map}E(\text{lim}_{\Delta} (P^RP)^{\bullet + 1}(F)) \to \text{lim}_{\Delta} E((P^RP)^{\bullet + 1}(F))\end{equation} is an equivalence. Now, since $E \simeq eP$ and $e$ is an equivalence, we deduce that \begin{equation}\label{P of Comparison Map}P(\text{lim}_{\Delta} (P^RP)^{\bullet + 1}(F)) \to \text{lim}_{\Delta} P((P^RP)^{\bullet + 1}(F))\end{equation} is an equivalence. Finally, we have an equivalence \begin{equation}\label{P Split Iso}P(F) \xrightarrow{\sim} \text{lim}_{\Delta} P((P^RP)^{\bullet + 1}(F))\end{equation} since $(P^RP)^{\bullet + 1}(F)$ is a $P$ split totalization. Since the maps $P(u(F))$, \labelcref{P of Comparison Map}, and \labelcref{P Split Iso} are compatible in the natural way, since \labelcref{P of Comparison Map} and \labelcref{P Split Iso} are equivalences we deduce that $P(u(F))$ is an equivalence. The functor $P$ is conservative, and so we deduce that $u(F)$ is an equivalence, as desired.
\end{proof}
\raggedbottom

\subsection{Identification of Comonads} \label{Identification of Comonads Subsection}We now prove \cref{NewBigTheorem}(4). To show \labelcref{IndCoh Counit Comonad Comparison} is an equivalence, it suffices to prove that \begin{equation}\label{Counit for ActI}\mathrm{act}_It^!  \simeq \mathrm{act}_I\mathrm{act}_I^RE_{\omega_{\LTd}}^R \xrightarrow{c_IE_{\omega_{\LTd}}^R } E_{\omega_{\LTd}}^R\end{equation} induced by the counit $c_I$ of the adjoint pair $(\mathrm{act}_I, \mathrm{act}_I^R)$ is an equivalence, since if we apply the functor $E_{\omega_{\LTd}}$ to \labelcref{Counit for ActI} and use essential uniqueness of adjoints we recover \labelcref{IndCoh Counit Comonad Comparison}. Since $\LTd$ is smooth, $\IndCoh(\LTd)$ is generated by the dualizing complex $\omega_{\LTd}$ and so it suffices to prove that the map \begin{equation}\label{Counit for ActI Evaluated at OmegaLTd} \mathrm{act}_It^!(\omega_{\LTd})  \simeq \mathrm{act}_I\mathrm{act}_I^RE_{\omega_{\LTd}}^R(\omega_{\LTd}) \xrightarrow{c_IE_{\omega_{\LTd}}(\omega_{\LTd})} E_{\omega_{\LTd}}^R(\omega_{\LTd})\end{equation} obtained from evaluating the natural transformation \labelcref{Counit for ActI} at $\omega_{\LTd}$ is an equivalence. To this end, we first observe that we have a commutative diagram \begin{equation}
\xymatrix@R+2em@C+2em{\LTd \times_{\LTd\sslash \Wext} \LTd \ar[r]^{(t, s, s)} \ar@/_1pc/[rr]_{s} \ar[d]^t & \LTd \times_{\LTd\sslash \Wext} \LTd \times \LTd \ar[r]^{\textcolor{white}{white}\mathrm{pr}_{\LTd}} & \LTd \ar[d]^{\quotientmapforcoarsequotient}\\  \LTd \ar[rr]^{\quotientmapforcoarsequotient} & & \LTd\sslash\Wext
  }\end{equation} where $\mathrm{pr}_{\LTd}$ is projection onto the rightmost factor. We deduce a natural isomorphism \begin{equation}\label{ActI is Pull Push}\mathrm{act}_It^!(\omega_{\LTd}) := t_*^{\IndCoh}(t, s, s)^!\mathrm{pr}_{\LTd}^! \simeq t_*^{\IndCoh}s^! \end{equation} by the functoriality of $!$-pullback.
  
  We use the notational shorthand $\mathfrak{c_{\mathrm{aff}}} := \LTd\sslash \Wext$. Observe that we have a canonical isomorphism $E_{\omega_{\mathfrak{c_{\mathrm{aff}}}}}(\quotientmapforcoarsequotient^!) := \quotientmapforcoarsequotient^!(\omega_{\mathfrak{c_{\mathrm{aff}}}}) \simeq \omega_{\LTd}$. In particular the unit map gives an isomorphism \begin{equation}\label{ER of OmegaLTd Is Push Pull Along Quotient}E^R_{\omega_{\LTd}}(\omega_{\LTd}) \simeq \mathcal{P}_{\quotientmapforcoarsequotient_*^{\IndCoh}}E^R_{\omega_{\mathfrak{c_{\mathrm{aff}}}}}(\omega_{\LTd}) \simeq \mathcal{P}_{\quotientmapforcoarsequotient_*^{\IndCoh}}E^R_{\omega_{\mathfrak{c_{\mathrm{aff}}}}}(E_{\omega_{\mathfrak{c_{\mathrm{aff}}}}}(s^!)) \xleftarrow{\sim} \mathcal{P}_{\quotientmapforcoarsequotient_*^{\IndCoh}}(\quotientmapforcoarsequotient^!) := \quotientmapforcoarsequotient^!\quotientmapforcoarsequotient_*^{\IndCoh}\end{equation} since $E_{\omega_{\mathfrak{c_{\mathrm{aff}}}}}$ is an equivalence of categories by \cref{Adjoints of Pulling Back by Adjoints and Evaluation at Monoidal Unit is Equivalence}(2). Thus, combining \labelcref{Counit for ActI Evaluated at OmegaLTd}  with \labelcref{ActI is Pull Push} and \labelcref{ER of OmegaLTd Is Push Pull Along Quotient} we obtain a map \[\eta: t_*^{\IndCoh}s^! \xrightarrow{} \quotientmapforcoarsequotient^!\quotientmapforcoarsequotient_*^{\IndCoh}.\] Tracing through the definitions, one can show that $\eta$ is the base change morphism. In particular, since $\IndCoh$ satisfies proper base change \cite[Chapter 3, Propoosition 2.2.2]{GaRoII} we deduce that this map is an equivalence, as desired. In a parallel manner, we obtain that \labelcref{DMod Counit Comonad Comparison} is an equivalence; the morphism analogous to the base change morphism above is an isomorphism since the functor on $\D(N^-_{\psi}\backslash G/N)$ obtained by $\psi$-averaging on the left side and then $N$-averaging on the right side is equivalent to the functor obtained from $N$-averaging on the left side and then $\psi$-averaging on the right side.
\subsection{Compact Generators of $\IndCoh(\LTd \times_{\LTd\sslash \Wext} \LTd)$}
In \cref{Summary of Companion Paper}(5), we recalled that $\CatTwTw$ has a canonical set of compact generators labeled by $\Wext$ given by the set $\{\delta_{\mathcal{D}} w : w \in \Wext\}$, where $\delta_{\mathcal{D}}$ denotes the monoidal unit of $\CatTwTw$. We obtain a similar description for the category $\IndCohx$, which we will use to prove \cref{NewBigTheorem}(5):

\begin{Proposition}\label{Compact Generators of IndCoh Category}
The category $\IndCohx$ has a canonical set of compact generators given by $\{\delta w : w \in \Wext\}$, where $\delta := i_*^{\text{IndCoh}}(\omega_{\LTd})$ and $i: \LTd \xhookrightarrow{} \LTd \times_{\LTd\sslash\Wext} \LTd$ the diagonal map, so that $\delta$ is the monoidal unit. 
\end{Proposition}

\begin{proof}
These objects are compact since the $\IndCoh$ pushforward by a closed embedding is a left adjoint with a continuous right adjoint, and thus preserves compact objects. We now show this set generates; fix a nonzero $\F \in \IndCohx$. There exists some finite subset $S \subseteq \Wext$ so that $i_S^!(\F)$ is nonzero. Note also that that for each finite $S \subseteq \Wext$, the map $\coprod_{w \in S}\LTd \to \Gamma_S$ is surjective at the level of geometric points, and so in particular, by \cite[Proposition 6.2.2]{GaRoI}, there exists some $w \in \Wext$ such that $i_w^!(\F)$ is nonzero. Therefore $\uHom(\omega_{\LTd}, i_w^!(\F)) \simeq \uHom(i_{w, *}^{\IndCoh}(\omega_{\LTd}), \F) \simeq \uHom(\delta w, \F)$ is nonzero. 
\end{proof}

\subsection{Proof of \cref{NewBigTheorem}(5) and \cref{Main Monoidal Equivalences}} \label{Proof of NewBigTheorem 5 and Main Monoidal Equivalences} By \cref{NewBigTheorem}(1) and \cref{NewBigTheorem}(2) all four of the functors $t_*^{\IndCoh, \mathrm{enh}}$ and $\mathsf{A}^{\psi, \mathrm{enh}}_!$ are equivalences when restricted to the eventually coconnective category. Moreover, by \cref{NewBigTheorem}(3), the functors $E_{\omega_{\LTd}}^{\mathrm{enh}}$ and $E_{\delta}^{\mathrm{enh}}$ are fully faithful. Using \cref{NewBigTheorem}(4), we may identify all of the corresponding comonads acting on $\IndCoh(\LTd)$. We deduce that the functors $\mathrm{act}_I$ and $\mathrm{act}_\D$ are equivalences of categories when restricted to the eventually coconnective objects, proving \cref{NewBigTheorem}(5). We also observe that these functors are $\Wext$-equivariant and send the monoidal units to the monoidal units. Therefore these equivalences of categories identify the full subcategories Karoubi generated by the monoidal unit and all of the objects (equivalent to) its $\Wext$-orbit. By \cref{Summary of Companion Paper}(5) and \cref{Compact Generators of IndCoh Category}, the categories $\CatTwTw$ and $\IndCoh(\LTd \times_{\LTd\sslash\Wext} \LTd)$ can be identified as the ind-completion of their respective subcategories of compact objects. Moreover, by \cite[Corollary 1.4.6]{DrGai2}, each of these subcategories of compact objects is Karoubi-generated by the objects in the $\Wext$-orbit of the monoidal unit. Therefore, we obtain our induced equivalence of categories by ind extension of the compact generators. Finally, the monoidality follows since the subcategory of $\Wext$-orbit of $\delta$ is a monoidal subcategory, and each of the functors in this equivalence of compact objects are monoidal. Since, $\Gamma_{\Wext} \simeq \LTd \times_{\LTd\sslash\Wext} \LTd$ essentially by construction of $\LTd\sslash\Wext$ (see \cite[Proposition 4.7]{GannonDescentToTheCoarseQuotientForPseudoreflectionAndAffineWeylGroups}) we deduce equivalences as in \labelcref{First Main Moniodal Equivalence}. The construction of equivalences as in \labelcref{Second Line of Equivalence} can be done similarly, or be obtained using de-equivariantization. \hfill \qedsymbol \qedsymbol
\begin{Remark}\label{Similar Equivalences Remark}
We expect similar methods yield equivalences of categories 
\raggedbottom
\[\LG\text{-mod}^N_{\text{nondeg}} \simeq \text{IndCoh}(\LTd\sslash W \times_{\LTd\sslash W^{\text{aff}}} \LTd/\characterlatticeforT)\]

\noindent compatible with the $T$ action and the $\IndCoh(\LTd\sslash W)$ action, and a monoidal equivalence
\raggedbottom
\[\mathcal{HC}_{\text{nondeg}} := \uEnd_G(\D(G)_{\text{nondeg}}^{G,w}) \simeq \text{IndCoh}(\LTd\sslash W \times_{\LTd\sslash W^{\text{aff}}} \LTd\sslash W).\]
\end{Remark}

\section{The Nondegenerate Horocycle Functor}\label{The Nondegenerate Horocycle Functor Section}
In this section, we construct a nondegenerate variant of the horocycle functor and show in \cref{Construction of W-Equivariance for Parabolic Restriction of Very Central Sheaves} that it can be used to equip $\text{Res}(\F)$ with a $W$-equivariance which descends to the coarse quotient $\LTd\sslash\Wext$ for $\F \in \D(G)^{G, \heartsuit}$ very central.

\subsection{The Nondegenerate Horocycle Functor}\label{Horocycle Functor Section}
We now construct a functor on $G \times G$ categories $\C$ which for $\C = \D(G)$ recovers the usual horocycle functor. To define such a functor, it suffices to define it in the \lq universal case\rq{} $\C = \D(G \times G)$. We consider the category $\D(G \times G)$ as a right $G \times G$ category and let $\Psi$ denote the composite functor 
\raggedbottom
\[\D(G \times G)^{\Delta_G} \xrightarrow{\text{oblv}^{\Delta_G}_{\Delta_B}} \D(G \times G)^{\Delta_B} \xrightarrow{\text{Av}_*^{N \times N}} \D(G/N \times G/N)^{\Delta_T}\]

\noindent where the group $\Delta_G$ denotes the diagonal copy of $G$ and the rightmost functor is induced by the averaging functor. Let $J^!$ denote the quotient functor $\D(G/N \times G/N) \to \D(G/N \times G/N)_{\text{nondeg}}$ which projects onto the nondegenerate subcategory, again taken with respect to the right action. Since this nondegenerate category is closed under the action of $T \times T$, we may equivalently view $J^!$ as a functor $\D(G/N \times G/N)^{\Delta_T} \to \D(G/N \times G/N)_{\text{nondeg}}^{\Delta_T}$. 

\begin{Theorem}\label{Claims About Horocycle Functor as G x G Categorical Functor}The functor $J^!\Psi: \D(G \times G)^{\Delta_G} \to \D(G/N \times G/N)^{\Delta_T}_{\text{nondeg}}$ lifts to a functor of $G \times G$ categories:
    \raggedbottom \[\tilde{\Psi}: \D(G \times G)^{\Delta_G} \to \D(G/N \times G/N)^{\Delta_{T \rtimes W}}_{\text{nondeg}}.\] Furthermore, fixing a simple coroot $\alpha$, we have the following: 
    
\begin{enumerate}
    \item The action of the Klein four group $\langle r_{\alpha} \times s_{\alpha} \rangle$ on $\D(G/N \times G/N)_{\text{nondeg}}^{\Delta_T\mathbb{G}_m^{\alpha}}$ is trivial. 
    \item The composite \raggedbottom
    \begin{equation}\label{Composite of TildePsi Forgetting and Averaging on Equivariant Category}\D(G \times G)^{\Delta_G} \xrightarrow{\text{Av}_*^{\mathbb{G}_m^{\alpha}}\text{oblv}_{\langle r_{\alpha} \rangle}^W\tilde{\Psi}} \D(G/N \times G/N)_{\text{nondeg}}^{\Delta_T\mathbb{G}_m^{\alpha}, \langle r_{\alpha} \rangle} \simeq \D(G/N \times G/N)^{\Delta_T\mathbb{G}_m^{\alpha}}_{\text{nondeg}} \otimes \text{Rep}\langle r_{\alpha} \rangle\end{equation}
    \noindent where the second equivalence is given by (2), lies entirely in the summand indexed by the trivial representation.
    \item If $\F \in \D(G)^G \simeq \D(G \times G)^{\Delta_G^{\ell} \times \Delta_G^r}$ is very central, then the sheaf 
    \raggedbottom
    \[\tilde{\Psi}(\F) \in \D(N\backslash G/N)_{\text{nondeg}}^{\Delta_{T \rtimes W}} \simeq \D((G/N \times G/N)/T)_{\text{nondeg}}^{\Delta_G, \Delta_W}\]
    
    \noindent has the property that the canonical $\langle r_{\alpha} \rangle$-representation on $\text{Av}_*^{\mathbb{G}_m^{\alpha}}\tilde{\Psi}(\F)$ is trivial. 
\end{enumerate}
\end{Theorem}

The final point of \cref{Claims About Horocycle Functor as G x G Categorical Functor} should be compared to the condition of descending to the coarse quotient given by the final point of \cref{New Various Conditions for Wext Equivariant Sheaf to Satisfy Coxteter Descent}. Note that we may identify $\text{hc} := \Psi^{\Delta_G^{\ell}}$.
 
\begin{proof} Note that the functor $\Psi$ itself is $G$-equivariant and so it suffices to construct a lift of $\text{hc}(\delta_{\Delta_{G}})$. However, $\text{hc}$ is monoidal. This fact as well known, and can be seen by identifying the monoidal functor
\begin{equation}\label{MonoidalityOfHorocycle}
\D(G)^G \simeq \uEnd_{G \times G}(\D(G)) \xrightarrow{\text{Av}^N_*} \uEnd_{G \times T}(\D(G/N)) \simeq \D(N\backslash G/N)^T
\end{equation}

\noindent with $\text{hc}$. Therefore, we see that we may identify $\Psi(\delta_{\Delta_{G}})$ (with its left $\Delta_G$-equivariance) in the category $\D(N\backslash G/N)^T$ with the monoidal unit equipped with its canonical $T$-equivariance. In particular, $J^!\Psi(\delta_{\Delta_{G}})$ may be identified with $J^!(\delta_1)$, where $\delta_1 \in \D(N\backslash G/N)^T$ is the monoidal unit. Under the equivalence \cref{Main Monoidal Equivalences}, the sheaf $J^!(\delta_1)$  corresponds to the pushforward $\Delta_*^{\IndCoh}(\omega_{\LTd}/\characterlatticeforT)$. In particular, this sheaf is equivariant with respect to the diagonal $W$-action. Thus we see that $J^!\Psi(\delta_{\Delta_G}) \in \D(G/N \times G/N)_{\text{nondeg}}^{T, \Delta_G} \simeq \D(N\backslash G/N)^T_{\text{nondeg}}$ may be equipped with a canonical $W$-equivariant structure, showing (1). Point (2) follows directly from \cref{Summary of Companion Paper}(2).

To show (3), note that we have identifications
\raggedbottom
\begin{equation}\label{Computation of Integral Kernel of Averaging Psi by Gmalpha}\text{oblv}_{\langle r_{\alpha} \rangle}^WJ^!\Psi(\delta_{\Delta_{G}}) \simeq \text{oblv}_{\langle r_{\alpha} \rangle}^WJ^!\text{Av}_*^{\mathbb{G}_m^{\alpha}}\Psi(\delta_{\Delta_{G}})\end{equation}

\noindent since the quotient functor is $T \times T$-equivariant. Under the Mellin transform, the sheaf $\text{Av}_*^{\mathbb{G}_m^{\alpha}}\Psi(\delta_{\Delta_{G}}) \in \D(T/\mathbb{G}_m^{\alpha})$ corresponds to the monoidal unit. In particular, \labelcref{Computation of Integral Kernel of Averaging Psi by Gmalpha} gives that the integral kernel of the $G \times G$-equivariant functor given by \labelcref{Composite of TildePsi Forgetting and Averaging on Equivariant Category} lies in the full $G \times G$-subcategory indexed by the trivial representation, establishing (3). Finally, (4) is a special case of (3).
\end{proof}

\newcommand{\tildehc}{\tilde{hc}}
\begin{Remark}
Note that the integral kernel of the composite $ch \circ hc$ on the level of nondegenerate categories canonically acquires a $W$-representation structure, and the sheaf $ch \circ hc(\delta)$ is known to be the Springer sheaf. This sheaf has endomorphisms which may be identified with the group ring of $W$--for a recent survey of this, see \cite{BZCHN}. Therefore we expect that, at least on the level of nondegenerate categories, the functor $\tilde{hc}: \D(G)_{\text{nondeg}}^{G} \to \D(N\backslash G/N)^{T \rtimes W}_{\text{nondeg}}$ is fully faithful. 

Of course, one did not need to pass to nondegenerate categories to obtain that the composite $ch \circ hc$ is given by convolution with the Springer sheaf. Therefore, one might hope that the functor hc factors through some subcategory $\D \xhookrightarrow{}\D(N\backslash G/N)^{T}$ which acquires a $W$-action, giving rise to a fully faithful functor $\tilde{hc}: \D(G)^G \xhookrightarrow{} \D^W$. We do not yet know what to make of this. 
\end{Remark}

\subsection{Construction of $W$-Equivariance for Parabolic Restriction of Very Central Sheaves}\label{Construction of W-Equivariance for Parabolic Restriction of Very Central Sheaves}
We now use the computations on the nondegenerate horocycle functor above to prove \cref{Parabolic Restriction of a Very Central Sheaf in Heart Has W-Equivariant Structure Descending to Coarse Quotient}.

\begin{proof}[Proof of \cref{Parabolic Restriction of a Very Central Sheaf in Heart Has W-Equivariant Structure Descending to Coarse Quotient}]
Because $T$ is connected, the forgetful functor induces an exact equivalence of abelian categories
\raggedbottom
\[\text{oblv}^T: \D(T)^{T \rtimes W, \heartsuit} \xrightarrow{\sim} \D(T)^{W, \heartsuit}\]

\noindent and so, in particular, to prove \cref{Parabolic Restriction of a Very Central Sheaf in Heart Has W-Equivariant Structure Descending to Coarse Quotient}, it suffices to exhibit the required $W$-equivariant structure on $\text{oblv}^T\text{Res}(\F)$. Since $\F$ is very central, the canonical map $i_{\ast, dR}\Res(\F) \xrightarrow{} \horocycleFunctor(\F)$ is an isomorphism by definition. In particular, the canonical map 
\raggedbottom
\begin{equation}\label{Parabolic Restriction of Very Central Is Horocycle at Nondegenerate Level}
J^!i_{\ast, dR}\text{oblv}^T\Res(\F) \xrightarrow{} J^!\text{oblv}^T\horocycleFunctor(\F)
\end{equation}

\noindent is also an isomorphism. Then, by taking the left diagonal $G$ invariants of the functor $\tilde{\Psi}$ of \cref{Claims About Horocycle Functor as G x G Categorical Functor}, we see that we may equip $J^!\text{oblv}^T\horocycleFunctor(\F)$ with a $W$-equivariant structure, and thus we may also equip $J^!i_{\ast, dR}\text{oblv}^T\Res(\F)$ with a $W$-equivariant structure. Furthermore, the functor 
\raggedbottom
\[\D(N\backslash G/N)_{\text{nondeg}} \xrightarrow{\Avpsi} \D(N_{\psi}^-\backslash G/N)\]

\noindent given by \textit{left} Whittaker averaging is $W$-equivariant, where $W$ acts on the domain diagonally and the codomain via the usual $W$-action.\footnote{Under \cref{Main Monoidal Equivalences}, this corresponds to the fact that the projection map $\LTd/\characterlatticeforT \times_{\LTd\sslash \Wext} \LTd/\characterlatticeforT \to \LTd/\characterlatticeforT$ is $W$-equivariant, where $W$ acts diagonally on the product and acts by the standard way on $\LTd/\characterlatticeforT$.} We lightly abuse notation and denote the composite of $\Avpsi$ with the $W$-equivariant equivalence of \cref{SupportOfWhittakerSheaves} by $\Avpsi$. Since this functor is a map of $W$-categories, we obtain a $W$-equivariant structure on 
\raggedbottom
\begin{equation}\label{Parabolic Restriction Is Whittaker Averaging of Inclusion of Parabolic Restriction}\Avpsi(J^!i_{\ast, dR}\text{oblv}^T\Res(\F)) \simeq  \Avpsi(i_{\ast, dR}\text{oblv}^T\Res(\F)) \simeq \text{oblv}^T\Res(\F)
\end{equation}

\noindent where the first equivalence is given by the definition of nondegeneracy and the second is given by direct computation.

We now show that the sheaf $\mathcal{R} := \text{oblv}^T(\text{Res}(\mathcal{F}))$ descends to the coarse quotient when equipped with the $W$-equivariance above. To see this, by \cref{New Various Conditions for Wext Equivariant Sheaf to Satisfy Coxteter Descent}(3) it suffices to show that, for every simple coroot $\gamma$, the $\langle r_{\gamma} \rangle$-representation on $\text{Av}_{\ast}^{\mathbb{G}_m^{\gamma}}(\mathcal{R})$ is trivial, where $\mathbb{G}_m^{\gamma}$ acts on $\D(N_{\psi}^-\backslash G/N)$ the right. However, we see that we have $W$-equivariant equivalences
\raggedbottom
\[\text{Av}_{\ast}^{\mathbb{G}_m^{\gamma}}(\mathcal{R}) \simeq \text{Av}_{\ast}^{\mathbb{G}_m^{\gamma}}\Avpsi(J^!i_{\ast, dR}(\mathcal{R})) \simeq \Avpsi\text{Av}_{\ast}^{\mathbb{G}_m^{\gamma}}(J^!i_{\ast, dR}(\mathcal{R})) \simeq \Avpsi\text{Av}_{\ast}^{\mathbb{G}_m^{\gamma}}(\text{oblv}^TJ^!\horocycleFunctor(\F))\]

\noindent where the first equivalence follows from \labelcref{Parabolic Restriction Is Whittaker Averaging of Inclusion of Parabolic Restriction}, the second equivalence follows from the fact that $\Avpsi$ is right $T$-equivariant, and the third equivalence follows from \labelcref{Parabolic Restriction of Very Central Is Horocycle at Nondegenerate Level} and the fact that $J^!$ is $T \times T$-equivariant. We therefore see that we have a $W$-equivariant equivalence
\raggedbottom
\begin{equation}\label{Averaging of OblvTRes of Very Central is Averaging of Nondegenerage Horocycle W Equivariantly}
\text{Av}_{\ast}^{\mathbb{G}_m^{\gamma}}(\mathcal{R}) \simeq \Avpsi\text{oblv}^T\text{Av}_{\ast}^{\mathbb{G}_m^{\gamma}}(J^!\horocycleFunctor(\F))\end{equation}

\noindent by base changing along the Cartesian diagram of quotient maps
\raggedbottom
\begin{equation*}
  \xymatrix@R+2em@C+2em{
   X \ar[d] \ar[r] & X/\mathbb{G}_m^{\gamma} \ar[d] \\
  X/\Delta_T  \ar[r] & X/\Delta_T\mathbb{G}_m^{\gamma}
  }
 \end{equation*}

\noindent where $X := G/N \times G/N$ Now, by \cref{Claims About Horocycle Functor as G x G Categorical Functor}(3) we see that the $\langle r_{\gamma} \rangle$-representation on $\Avpsi\text{oblv}^T\text{Av}_{\ast}^{\mathbb{G}_m^{\gamma}}(J^!\horocycleFunctor(\F))$ is trivial. Thus by \labelcref{Averaging of OblvTRes of Very Central is Averaging of Nondegenerage Horocycle W Equivariantly} we obtain the $\langle r_{\gamma} \rangle$-representation on $\text{Av}_{\ast}^{\mathbb{G}_m^{\gamma}}(\mathcal{R})$ is trivial, as desired. 
\end{proof}

    \appendix

\section{Mellin Transform (With Germ\'an Stefanich)}\label{Mellin Transform Appendix}
In this appendix, written jointly with Germ\'an Stefanich, we discuss the foundations of the Mellin transform in the higher categorical setting we use above. In \cref{Derivation of Mellin Transorm Subsection}, we give the construction of the Mellin transform in the higher categorical setting, and we exhibit a functoriality of the Mellin transform we use above in \cref{Functoriality of Mellin Transform Subsection}. Finally, in \cref{Symmetric Monoidality Subsubsection}, we upgrade this Mellin transform to an equivalence of symmetric monoidal DG categories.
    \subsection{Derivation of Mellin Transform}\label{Derivation of Mellin Transorm Subsection}
Notice that the following diagram is Cartesian \begin{equation}\label{Group Action Cartesian Diagram for LTd and Character Action}
\xymatrix@R+2em@C+2em{\characterlatticeforT \times \LTd  \ar[r]^{\text{act}} \ar[d]^{\text{proj}} & \LTd \ar[d]^{\quotientMapFromLTdToQuotientByCharacterLattice}\\  \LTd \ar[r]^{\quotientMapFromLTdToQuotientByCharacterLattice} & \LTd/\characterlatticeforT
  }\end{equation} which shows that the map $\quotientMapFromLTdToQuotientByCharacterLattice$ is ind-proper. Therefore, the functor $\quotientMapFromLTdToQuotientByCharacterLattice^{\IndCoh}_*$ (defined via an identical procedure below \cite[Corollary 4.14]{GannonDescentToTheCoarseQuotientForPseudoreflectionAndAffineWeylGroups}) is left adjoint to $\quotientMapFromLTdToQuotientByCharacterLattice^{!}$. Moreover, since the functor $\quotientMapFromLTdToQuotientByCharacterLattice^{!}$ is conservative, the sheaf $\quotientMapFromLTdToQuotientByCharacterLattice_*^{\IndCoh}(\LTd)$ is a compact generator of $\IndCoh(\LTd/\characterlatticeforT)$ and therefore gives an equivalence of categories \begin{equation}\label{Compact Generator Equivalence of Categories Prelude to Mellin Transform}
       \IndCoh(\LTd/\characterlatticeforT) \simeq \uEnd_{\IndCoh(\LTd/\characterlatticeforT)}(\quotientMapFromLTdToQuotientByCharacterLattice_*^{\IndCoh}(\LTd))\text{-mod}.
  \end{equation}
  Notice that, as a graded vector space, by base changing along \labelcref{Group Action Cartesian Diagram for LTd and Character Action}, we may identify \[\uEnd_{\IndCoh(\LTd/\characterlatticeforT)}(\quotientMapFromLTdToQuotientByCharacterLattice_*^{\IndCoh}(\LTd)) \simeq \uHom_{\IndCoh(\LTd)}(\omega_{\LTd}, \text{proj}_*^{\IndCoh}(\omega_{\characterlatticeforT \times \LTd}))\] \[\simeq \oplus_{\characterlatticeforT}\uEnd_{\IndCoh(\LTd)}(\omega_{\LTd}) \simeq \oplus_{\characterlatticeforT}\Symt\] which is in particular a discrete vector space concentrated in a single cohomological degree. One may therefore check that this equivalence of graded vector spaces upgrades to an equivalence of $k$-algebras \[\uEnd_{\IndCoh(\LTd/\characterlatticeforT)}(\quotientMapFromLTdToQuotientByCharacterLattice_*^{\IndCoh}(\LTd)) \simeq \Gamma(\mathcal{D}_T)\] and so the equivalence given by \labelcref{Compact Generator Equivalence of Categories Prelude to Mellin Transform} yields equivalences \begin{equation}\label{Categories of Mellin Transform by Compact Generators Definition}\IndCoh(\LTd/\characterlatticeforT) \simeq \Gamma(\mathcal{D}_T)\text{-mod} \simeq \D(T)\end{equation} whose composite we denote by $\text{FMuk}_T$ and refer to as the \textit{Mellin transform}. When the associated torus $T$ is clear from context, we may also denote this transformation by $\text{FMuk}$. Equipping $\IndCoh(\LTd/\characterlatticeforT)$ with a $t$-structure as in \cite{GannonDescentToTheCoarseQuotientForPseudoreflectionAndAffineWeylGroups}, we have that $\quotientMapFromLTdToQuotientByCharacterLattice^{\IndCoh}_*(\omega_{\LTd})$ is concentrated in cohomological degree $-\text{dim}(T)$ and so the Mellin transform is $t$-exact up to cohomological shift.  

 \subsection{Functoriality of Mellin Transform}\label{Functoriality of Mellin Transform Subsection}
We will use the following fact:
\begin{Proposition}
    Assume $1 \to S \to T \xrightarrow{\phi} T/S \to 1$ is a short exact sequence of split algebraic tori. Then we have a canonical isomorphism of functors \begin{equation}\label{Functoriality of Mellin Transform}
\xymatrix@R+2em@C+2em{\D(T) \ar[r]^{\phi_{*, dR}} \ar[d]^{\text{FMuk}_T} & \D(T/S) \ar[d]^{\text{FMuk}_{T/S}}\\ \IndCoh(\LTd/\characterlatticeforT) \ar[r]^{\iota^!} & \IndCoh((\LT/\mathfrak{s})^*/X^{\bullet}(T/S))
  }
  \end{equation} where $\iota: (\LT/\mathfrak{s})^*/X^{\bullet}(T/S) \to \LTd/\characterlatticeforT$ is the induced map given by inclusion. 
\end{Proposition}

\begin{proof}
    The above short exact sequence splits, and therefore by induction we may assume that $S = \mathbb{G}_m$. In this case, $\phi_{*, dR}(\D_T)$ is equivalently given by the complex \[0 \to \Gamma(\D_T) \xrightarrow{\partial_S} \Gamma(\D_T) \to 0\] so that we may identify $\phi_{*, dR}(\D_T)$ as a direct sum of $\mathbb{Z}$-many copies of $\Gamma(\D_{T/S})$. We claim also that $(\text{FMuk}_{T/S}^{-1}\circ \iota^!\circ \text{FMuk}_T)(\D_T)$ maps to an isomorphic $\Gamma(\D_T)$-module; this can be checked explicitly since $\phi_{*, dR}(\D_T)$ lies in the heart of a $t$-structure. To see this, notice that the following diagram is Cartesian \begin{equation}\label{Cartesian Diagram for Functoriality of Mellin Transform}
\xymatrix@R+2em@C+2em{\LTd \times_{\mathfrak{s}^*} X^{\bullet}(S) \ar[r]^{\quotientMapFromLTdToQuotientByCharacterLattice} \ar[d] & (\LTd \times_{\mathfrak{s}^*} X^{\bullet}(S))/\characterlatticeforT \ar[d]\\  \LTd \ar[r]^{\quotientMapFromLTdToQuotientByCharacterLattice} & \LTd/\characterlatticeforT
  }\end{equation} where the unlabelled vertical arrows are induced by inclusion. Notice also that we have a canonical isomorphism $\LTd \times_{\mathfrak{s}^*} X^{\bullet}(S) \simeq (\mathfrak{t}/\mathfrak{s})^* \times X^{\bullet}(S)$ which is equivariant with respect to the action of $X^{\bullet}(T) \simeq X^{\bullet}(S) \times X^{\bullet}(T/S)$. This in particular shows that we have a Cartesian diagram  \begin{equation}
\xymatrix@R+2em@C+2em{(\mathfrak{t}/\mathfrak{s})^* \times X^{\bullet}(S) \ar[r] \ar[d] & (\mathfrak{t}/\mathfrak{s})^*/X^{\bullet}(T/S)\ar[d]\\  \LTd \ar[r]^{\quotientMapFromLTdToQuotientByCharacterLattice} & \LTd/\characterlatticeforT
  }\end{equation} and by base changing along this Cartesian diagram we obtain our desired claim. 
\end{proof} 

\subsection{Symmetric Monoidality}\label{Symmetric Monoidality Subsubsection}
The Mellin transform of \labelcref{Categories of Mellin Transform by Compact Generators Definition} provides an equivalence of categories $\D(T) \simeq \IndCoh(\LTd/\characterlatticeforT)$ which is $t$-exact up to a shift. This induces an equivalence of abelian categories, which moreover is well known to be symmetric monoidal--in the abelian categorical setting, this can be checked explicitly. On the other hand, in the higher categorical context, equipping a functor between symmetric monoidal $\infty$-categories with a symmetric monoidal structure requires an infinite amount of additional structure. The entirety of \cref{Symmetric Monoidality Subsubsection} is devoted to the following theorem, which provides this upgraded structure: 

\begin{Theorem}\label{theorem FM sym monoidal}
    The Mellin transform $\FourierMukai_T$ can be upgraded to an equivalence of symmetric monoidal categories with a $W$-action. 
\end{Theorem}

We will obtain \cref{theorem FM sym monoidal} as a consequence of a general uniqueness principle for symmetric monoidal structures on derived  $\infty$-categories. To formulate it we first need to introduce some notation.

\begin{Notation}
Denote by $\Groth$ the category of Grothendieck abelian categories and colimit preserving functors, and by $\Groth_\proj$ the subcategory of $\Groth$ on those Grothendieck abelian categories with enough projectives and functors which preserve projective objects. We denote by $\DGroth_\proj$ the category defined informally as follows:
\begin{itemize}
\item Objects of $\DGroth_\proj$ are derived categories of Grothendieck abelian categories with enough projectives.
\item A morphism $f: \ccal \rightarrow \dcal$ in $\DGroth_\proj$ is a colimit preserving functor that sends projective objects of $\ccal^\heartsuit$ to projective objects of $\dcal^\heartsuit$.
\end{itemize}
\end{Notation}

In addition to having morphisms, $\Groth_\proj$ has operations of  arity $n$ for any nonnegative integer $n$. Namely, for each finite family of source objects $\ccal_1, \ldots, \ccal_n$ and target object $\ccal$ an operation $\lbrace \ccal_1, \ldots, \ccal_n \rbrace \rightarrow \ccal$ is a functor
\[
f: \ccal_1 \times \ccal_2 \times \ldots \times \ccal_n \rightarrow \ccal
\]
with the following properties:
\begin{itemize}
\item $f$ is colimit preserving on each variable.
\item For each sequence of projective objects $X_i$ in $\ccal_i$ the object $f(X_1, \ldots, X_n)$ is projective.
\end{itemize}
We may summarize the situation by saying that $\Groth_\proj$ has the structure of an operad. 

Similarly, $\DGroth_\proj$ has the structure of an operad, where an operation $\lbrace \ccal_1, \ldots, \ccal_n \rbrace \rightarrow \ccal$ is a functor
\[
f: \ccal_1 \times \ccal_2 \times \ldots \times \ccal_n \rightarrow \ccal
\]
with the following properties:
\begin{itemize}
\item $f$ is colimit preserving on each variable.
\item For each sequence of projective objects $X_i$ in $\ccal^\heartsuit_i$ the object $f(X_1, \ldots, X_n)$ belongs to $\ccal^\heartsuit$ and is projective.
\end{itemize}

We note that given an operation $f$ in $\DGroth_\proj$ as above there is a corresponding operation $f^\heartsuit$ in $\Groth_\proj$  given by the following composition:
\[
 \ccal^\heartsuit_1 \times \ldots \times \ccal^\heartsuit_n \hookrightarrow \ccal^{\leq 0}_1 \times \ldots  \times \ccal^{\leq 0}_n \xrightarrow{f} \ccal^{\leq 0} \xrightarrow{H^0} \ccal^\heartsuit
\]
This forms part of a morphism of operads $\DGroth_\proj \rightarrow \Groth_\proj$. We are now ready to state the basic assertion that allows us to lift structures from the abelian setting to the derived setting:

\begin{Theorem}\label{theorem equivalence operads}
\hspace{1cm}
\begin{enumerate}[\normalfont (1)]
\item $\Groth_\proj$ and $\DGroth_\proj$ are symmetric monoidal categories.
\item The assignment $\ccal \mapsto \ccal^\heartsuit$ provides a symmetric monoidal equivalence $\DGroth_\proj = \Groth_\proj$. 
\end{enumerate}
\end{Theorem}

Before giving the proof of \cref{theorem equivalence operads}, we indicate how it can be used to deduce \cref{theorem FM sym monoidal}. In what follows it will be convenient to shift the t-structure on $\IndCoh(\LTd/\characterlatticeforT)$ so that $\FourierMukai_T$ becomes t-exact.

\begin{Lemma}\label{lemma tensor projectives}
    For for any pair of projective objects $\F, \mathcal{G} \in \D(T)^{\heartsuit}$  the convolution $\F \star \mathcal{G}$ is a projective object of $ \D(T)^{\heartsuit}$. Similarly, for any pair of projective objects $M, N \in \IndCoh(\LTd/\characterlatticeforT)^{\heartsuit}$, the tensor product $M \mathop{\otimes}\limits^{!} N$ is a  projective object of $ \IndCoh(\LTd/\characterlatticeforT)^{\heartsuit}$. 
\end{Lemma}

\begin{proof}

Because the following diagram commutes \begin{equation}\xymatrix@R+2em@C+2em{T \times T \ar[r] \ar[d] & T \ar[d]\\
 T_{dR} \times T_{dR} \ar[r] & T_{dR}
  }\end{equation} where the horizontal maps induced by multiplication and the vertical maps are the canonical maps, we see that in the category $\D(T)$, if $\mathcal{G}$ is the is the compact generator $(T \to T_{dR})_*^{\IndCoh}(\omega_T)$, then $\mathcal{G} \star \mathcal{G}$ is an infinite direct sum of copies of $\mathcal{G}$. The projection formula (\cite[Section 2.1.8]{GaRoII}) similarly implies that if $\mathcal{G}' := \quotientMapFromLTdToQuotientByCharacterLattice_*^{\IndCoh}(\omega_{\LTd})$ then we have an isomorphism of $\mathcal{G}' \mathop{\otimes}\limits^{!}\mathcal{G}'$ with \[\quotientMapFromLTdToQuotientByCharacterLattice_*^{\IndCoh}(\omega_{\LTd}) \mathop{\otimes}\limits^{!}\quotientMapFromLTdToQuotientByCharacterLattice_*^{\IndCoh}(\omega_{\LTd}) \simeq \quotientMapFromLTdToQuotientByCharacterLattice_*^{\IndCoh}(\omega_{\LTd}\mathop{\otimes}\limits^{!}\quotientMapFromLTdToQuotientByCharacterLattice^!\quotientMapFromLTdToQuotientByCharacterLattice_*^{\IndCoh}(\omega_{\LTd})) \simeq \quotientMapFromLTdToQuotientByCharacterLattice_*^{\IndCoh}\quotientMapFromLTdToQuotientByCharacterLattice^!\quotientMapFromLTdToQuotientByCharacterLattice_*^{\IndCoh}(\omega_{\LTd})\] and so $\mathcal{G}' \mathop{\otimes}\limits^{!}\mathcal{G}'$ is also an infinite direct sum of copies of $\mathcal{G}'$, by base change along the diagram \labelcref{Group Action Cartesian Diagram for LTd and Character Action}. Now the claims follow since both $\mathcal{G}$ and $\mathcal{G}'$ are projective generators of the associated abelian categories (which we can see by and \cite{GaiRozCrystals} and the fact that $\quotientMapFromLTdToQuotientByCharacterLattice^!$ is conservative by base change respectively) so any projective object is a direct summand of a direct sum of $\mathcal{G}$ or $\mathcal{G}'$ respectively.    
\end{proof}

\begin{proof}[Proof of \cref{theorem FM sym monoidal}]
It follows from \cref{lemma tensor projectives} that   $\D(T)$ is a nonunital commutative $\operatorname{Vect}$-algebra in $\DGroth_\proj$. Note that it has an action of $W$ induced from the action of $W$ on $T$. Similarly, $\IndCoh(\LTd/\characterlatticeforT)$ is also a  nonunital commutative $\operatorname{Vect}$-algebra in $\DGroth_\proj$ with an action of $W$. By \cref{theorem equivalence operads} the usual $W$-equivariant symmetric monoidal structure on $\FourierMukai_T^\heartsuit$ may be upgraded to a $W$-equivariant nonunital symmetric monoidal  structure on $\FourierMukai_T$. This equivalence admits a unique unital symmetric monoidal extension by virtue of \cite[Theorem 5.4.3.5]{LuHA}.
\end{proof}
 
We now turn to the proof of \cref{theorem equivalence operads}.

\begin{Notation}
For each object $\ccal$ of $\Groth_\proj$ we denote by $\ccal_\proj$ the full subcategory of $\ccal$ on the projective objects. 
\end{Notation}

\begin{Lemma}\label{lemma univ property derived}
Let $\ccal$ be an object of $\DGroth_\proj$. Then:
\begin{enumerate}[\normalfont (1)]
\item For every category with small colimits $\dcal$, restriction to $\ccal^\heartsuit_\proj$ provides an equivalence between the category  $\Funct^L(\ccal^{\leq 0}, \dcal)$ of colimit preserving functors $\ccal^{\leq 0} \rightarrow \dcal$ and the category $\Funct^\oplus(\ccal^\heartsuit_\proj, \dcal)$ of small coproduct preserving functors $\ccal^\heartsuit_\proj \rightarrow \dcal$.
\item For every stable category with small colimits $\dcal$, restriction to $\ccal^\heartsuit_\proj$ provides an equivalence $\Funct^L(\ccal, \dcal) = \Funct^\oplus(\ccal^\heartsuit_\proj, \dcal)$.
\end{enumerate}
\end{Lemma}
\begin{proof}
Part (2) follows from part (1) since $\ccal$ is the stabilization of $\ccal^{\leq 0}$. To prove part (1) we apply the results from \cite{DAGVIII} section 4.2. The category $\ccal^\heartsuit_\proj$ is a socle in the sense of definition 4.2.9. The lemma follows from a combination of corollary 4.2.14 and proposition 4.2.15. 
\end{proof}

\begin{Lemma}\label{lemma univ prop abelian}
Let $\ccal$ be an object of $\Groth_\proj$. Then for every $(1,1)$-category with small colimits $\dcal$, restriction to $\ccal_\proj$ provides an equivalence $\Funct^L(\ccal, \dcal) = \Funct^\oplus(\ccal_\proj, \dcal) $.
\end{Lemma}
\begin{proof}
Follows by applying part (1) of \cref{lemma univ property derived} to the derived category of $\ccal$.
\end{proof}

\begin{proof}[Proof of \cref{theorem equivalence operads}]
We first show that the assignment $\ccal \mapsto \ccal^\heartsuit$ is an equivalence of operads. Since it is surjective, it is enough to show that for every sequence $\ccal_1, \ldots, \ccal_n, \ccal$ of objects of $\DGroth_\proj$ passage to hearts induces an equivalence 
\[
\Hom_{\DGroth_\proj}(\lbrace \ccal_1, \ldots, \ccal_n \rbrace, \ccal) = \Hom_{\Groth_\proj}(\lbrace \ccal_1^\heartsuit, \ldots, \ccal^\heartsuit_n \rbrace, \ccal^\heartsuit).
\]
The case $n = 0$ is clear: in both cases a $0$-ary operation consists of a projective object of $\ccal^\heartsuit$. Assume now that $n > 0$. Then applying \cref{lemma univ prop abelian} to the case when $\dcal$ is the category of functors $\ccal^\heartsuit_2 \times \ldots \times \ccal^\heartsuit_n \rightarrow \ccal^\heartsuit$ which preserve colimits in each variable we see that restriction along the inclusion $(\ccal^\heartsuit_1)_\proj \rightarrow \ccal^\heartsuit_1$ provides an equivalence between $\Hom_{\Groth_\proj}(\lbrace \ccal_1^\heartsuit, \ldots, \ccal^\heartsuit_n \rbrace, \ccal^\heartsuit)$ and the space of functors
\[
f: (\ccal^\heartsuit_1)_\proj \times \ccal^\heartsuit_2 \times \ccal^\heartsuit_3 \times \ldots \times \ccal^\heartsuit_n \rightarrow \ccal^\heartsuit
\]
with the following properties:
\begin{itemize}
\item $f$ preserves coproducts in the first variable.
\item $f$ preserves colimits in the variables $2, \ldots, n$.
\item For every sequence of projective objects $X_i$ in $\ccal^\heartsuit_i$ the object $f(X_1, \ldots, X_n)$ is projective.
\end{itemize}
Applying this reasoning inductively we conclude that \[\Hom_{\Groth_\proj}(\lbrace \ccal_1^\heartsuit, \ldots, \ccal^\heartsuit_n \rbrace, \ccal^\heartsuit)\] is equivalent to the space of functors
\[
(\ccal^\heartsuit_1)_\proj \times (\ccal^\heartsuit_2)_\proj \times \ldots \times (\ccal^\heartsuit_n)_\proj \rightarrow \ccal^\heartsuit_\proj
\]
which preserve coproducts in each variable.

Similarly, an inductive application of part (2) of \cref{lemma univ property derived} shows that $\Hom_{\DGroth_\proj}(\lbrace \ccal_1, \ldots, \ccal_n \rbrace, \ccal) $ is also equivalent to the above space. This concludes the proof that the map $\DGroth_\proj \rightarrow \Groth_\proj$ is an equivalence of operads. It remains to show that these are symmetric monoidal categories.

Equip $\Groth$ with the structure of operad where an operation $\lbrace \ccal_1, \ldots, \ccal_n \rbrace \rightarrow \ccal$  is a functor
\[
f: \ccal_1 \times \ldots \times \ccal_n \rightarrow \ccal
\]
which preserves colimits in each variable. This is in fact a symmetric monoidal category, and the inclusion $\Groth \rightarrow \Pr^L$ preserves tensor products \cite[Corollary C.5.4.19]{SAG}. We may regard $\Groth_\proj$ as a suboperad of $\Groth$. The unit of $\Groth$ is the category of abelian groups, which belongs to $\Groth_\proj$. To finish the proof it will suffice to show that if $\ccal$ and $\dcal$ are objects of $\Groth_\proj$ then the tensor product $\ccal \otimes \dcal$ (computed in $\Groth$) is still a tensor product in $\Groth_\proj$. It is enough for this to prove that if $X$ and $Y$ are projective objects of $\ccal$ and $\dcal$ then $X \otimes Y$ is a projective object of $\ccal \otimes \dcal$. 

The Yoneda embedding provides an equivalence between $\ccal \otimes \dcal$ and the category of limit preserving functors from  $(\ccal \otimes \dcal)^\op$ into the category of sets. This is equivalent to the category of functors $\ccal^\op \times \dcal^\op \rightarrow  \operatorname{Set}$ which preserve limits in each variable.  Applying once again \cref{lemma univ prop abelian} we obtain an equivalence between $\ccal \otimes \dcal$ and the category $\Funct^\prod(\lbrace \ccal_\proj^\op, \dcal_\proj^\op \rbrace, \operatorname{Set})$ of functors $\ccal_\proj^\op \times \dcal_\proj^\op \rightarrow \operatorname{Set}$ which preserve products in each variable.  The functor $\ccal \otimes \dcal \rightarrow \operatorname{Set}$ corepresented by $X \otimes Y$ corresponds under this dictionary to the functor 
\[
\operatorname{ev}_{(X, Y)}: \Funct^{\prod}(\lbrace \ccal_\proj^\op, \dcal_\proj^\op \rbrace, \operatorname{Set}) \rightarrow \operatorname{Set}  
\]
of evaluation at $(X,Y)$. Our goal is to show that $\operatorname{ev}_{(X, Y)}$ preserves geometric realizations. To do so it suffices to prove that $ \Funct^{\prod}(\lbrace \ccal_\proj^\op, \dcal_\proj^\op \rbrace, \operatorname{Set})$ is closed under geometric realizations in the category $ \Funct(\lbrace \ccal_\proj^\op, \dcal_\proj^\op \rbrace, \operatorname{Set})$ of functors $\ccal_\proj^\op \times \dcal_\proj^\op \rightarrow \operatorname{Set}$.

Let $\operatorname{Ab}$ be the category of abelian groups. We have a commutative square of categories
\[
\begin{tikzcd}
\Funct^{\prod}(\lbrace \ccal_\proj^\op, \dcal_\proj^\op \rbrace, \operatorname{Ab})  \arrow{r}{} \arrow{d}{} & \Funct^{\prod}(\lbrace \ccal_\proj^\op, \dcal_\proj^\op \rbrace, \operatorname{Set}) \arrow{d}{} \\
\Funct(\lbrace \ccal_\proj^\op, \dcal_\proj^\op \rbrace, \operatorname{Ab})  \arrow{r}{}  & \Funct(\lbrace \ccal_\proj^\op, \dcal_\proj^\op \rbrace, \operatorname{Set})
\end{tikzcd}
\]
where the horizontal arrows are induced by the forgetful functor $\operatorname{Ab} \rightarrow \operatorname{Set}$ and the vertical arrows are the inclusions. The top horizontal arrow is an equivalence since $\ccal^\op_{\proj}$ and $\dcal^\op_\proj$ are additive. The bottom horizontal arrow preserves geometric realizations since these are preserved by the forgetful functor $\operatorname{Ab} \rightarrow \operatorname{Set}$. We may therefore reduce to showing that left vertical arrow preserves geometric realizations. This follows from the fact that products in $\operatorname{Ab}$ are exact.
\end{proof}

\printbibliography

    \end{document}